\documentclass[a4paper,final,fleqn]{article}
	
\usepackage[lm,grey]{janm}
\usepackage{math}

\usepackage[numbers]{natbib}
\def\@biblabel#1{#1.}
 
\newcommand{\ld}{\text{\tiny\ensuremath{\bullet}}}
\newcommand{\dDl}{\Dl^{\ld}}
\newcommand{\Uf}{\mathfrak{U}}

\newcommand{\nn}[1]{\n{#1}_{0}}
\newcommand{\les}{\lesssim}

\title{New estimates of the nonlinear Fourier transform for the defocusing \nls equation}
\author{Jan-Cornelius Molnar}
\date{March 5, 2014}

\begin{document}

\maketitle

\begin{abstract}
The defocusing NLS-equation $\mathrm{i} u_t = -u_{xx} + 2|u|^2u$ on the circle admits a global nonlinear Fourier transform, also known as Birkhoff map, linearising the NLS-flow. The regularity properties of $u$ are known to be closely related to the decay properties of the corresponding nonlinear Fourier coefficients. In this paper we quantify this relationship by providing two sided polynomial estimates of all integer Sobolev norms $\|u\|_m$, $m\ge 0$, in terms of the weighted norms of the nonlinear Fourier transformed.
\end{abstract}

\section{Introduction}
\label{s:intro}

Consider the \emph{defocusing nonlinear Schrödinger equation}
\[
  \ii\partial_tu = -\partial_x^2u + 2\abs{u}^2u,\qquad x\in\T,
\]
on the circle $\T = \R/\Z$ with $u$ taken from $L^2 \defl L^2(\T,\C)$. As is well known, the \nls-equation can be written as an \emph{infinite dimensional Hamiltonian system}. We introduce the phase space $L^2_c\defl L^2\times L^2$ with elements $\ph=(\phm,\php)$ and define the inner product
\[
  \lin{\ph,\psi} \defl \int_{\T} \ph_{+}\ob{\psi_{+}} + \ph_{-}\ob{\psi_{-}}\,\dx,
\]
which makes $L^{2}_{c}$ a Hilbert space. The associated norm is denoted by $\n{\ph}_{0}$. Further, we endow this space with the Poisson structure given by
\[
  \{F,G\} 
  \defl 
  -\ii\int_{\T} (\partial_{\phm} F\, \partial_{\php} G 
    - \partial_{\php}F\,\partial_{\phm} G)\,\dx.
\]
Here $\partial_{\phm} F$ and $\partial_{\php} F$ denote the components of the $L^2$-gradient $\partial_{\ph} F$ of a $C^1$-functional $F$. The Hamiltonian system with Hamiltonian
\[
  \Hc(\phm,\php) \defl \int_{\T} (\partial_{x}\phm\partial_{x}\php + \phm^{2}\php^{2})\,\dx
\]
is then given by
\[
  \ii\partial_{t} (\phm,\php) = (\partial_{\php}\Hc,-\partial_{\phm}\Hc),
\]
and the defocusing \nls is obtained by restriction to the invariant subspace of real type states
\[
  L^2_r
  \defl
  \setdef{\ph \in L_c^2}{\ph^* = \ph},\qquad \ph^* = (\ob\php,\ob\phm).
\]
Indeed, with $\ph = (u,\ob{u})$ the defocusing \nls-equation becomes
\[
  \ii\partial_{t}u = \ii\setd{u,\Hc} = \partial_{\ob{u}}\Hc(u,\ob{u}).
\]

After the \kdv-equation, the \nls-equation was the second evolution equation known to be \emph{completely integrable} by the inverse scattering method~\cite{Zakharov:1975ft}. In fact, according to \cite{Grebert:2014iq} (cf. also \cite{Battig:1995jr,Battig:1993vu,McKean:1997ka}), the defocusing \nls-equation is integrable in the strongest possible sense meaning that it admits global \emph{Birkhoff coordinates} $(x_{n},y_{n})_{n\in\Z}$.

To state our main result, let $\sum_{n\in\Z} (\ph_{2n}^{-}\e^{-\ii 2n\pi x},\ph_{2n}^{+}\e^{\ii 2n\pi x})$ denote the Fourier series of $\ph=(\phm,\php)\in L_{c}^{2}$ and introduce for any $s\ge 0$ the Sobolev norm $\n{\ph}_{s}$ given by
\[
  \n{\ph}_{s}^{2} 
  	\defl \sum_{n\in\Z} \lin{2n\pi}^{2s}(\abs{\ph_{2n}^{-}}^{2}+\abs{\ph_{2n}^{+}}^{2}),
  		  \qquad \lin{x} \defl 1+\abs{x}.
\]
The space of all $\ph\in L_{c}^{2}$ with $\n{\ph}_{s} < \infty$ is denoted $H^{s}_{c}$, and $H_{r}^{s} \defl L^{2}_{r}\cap H_{c}^{s}$. Furthermore, let us introduce the model space
\[
  h_{r}^{s} \defl
   \setdef{(x,y)=(x_{n},y_{n})_{n\in\Z}}
          {\n{(x,y)}_{s} \defl (\n{x}_{s}^{2}+\n{y}_{s}^{2})^{1/2} < \infty},
\]
where the norm $\n{x}_{s}$ is defined by
\[
  \n{x}_{s}^{2} \defl \sum_{n\in\Z} \lin{2n\pi}^{2s}\abs{x_{n}}^{2}.
\]
This space is endowed with the canonical Poisson structure $\{x_{n},y_{m}\} = -\dl_{n,m}$ while all other brackets vanish.

The \emph{Birkhoff map} $\ph\mapsto (x_{n},y_{n})_{n\in\Z}$ is a bi-analytic, canonical diffeomorphism $\Om\colon H^{0}_{r}\to h^{0}_{r}$, whose restriction $\Om\colon H^{m}_{r}\to h_{r}^{m}$ is again bi-analytic for any integer $m\ge 1$. On $h^{1}_{r}$ the transformed \nls Hamiltonian $\Hc\circ\Om^{-1}$ is a real analytic function of the actions $I_{n} = (x_{n}^{2}+y_{n}^{2})/2$ alone and the \nls evolution takes the particularly simple form
\[
  \dot x_{n} = -\om_{n} y_{n},\quad
  \dot y_{n} =  \om_{n} x_{n},\qquad
  \om_{n} \defl \partial_{I_{n}} \Hc.
\]
One may thus think of $\Om$ as a nonlinear Fourier transform for the defocusing \nls-equation.
Remarkably, the derivative $\ddd_{0}\Om$ of $\Om$ at the origin is the Fourier transform and on $L_{r}^{2}$, as for the Fourier transform,
\[
  \n{\Om(\ph)}_{0} = \n{\ph}_{0},
\]
which we also refer to as Parseval's identity -- cf. e.g. \cite{McKean:1997ka,Grebert:2014iq}.
Our main result says that for higher order Sobolev norms the following version of Parseval's identity holds for the nonlinear map $\Om$.

\begin{theorem}
\label{b-est}
For any integer $m\ge 1$ there exist absolute constants $c_{m}$, $d_{m} > 0$ such that
the restriction of $\Om$ to $H_{r}^{m}$ satisfies the two sided estimates
\[
  \emph{ (i)}\quad
  \n{\Om(\ph)}_{m} \le c_{m}\bigl(
  	\n{\ph}_{m} + (1+\n{\ph}_{1})^{2m}\n{\ph}_{0}	\bigr),
\]
\[
  \emph{(ii)}\quad 
  \n{\ph}_m \le d_{m}\bigl(
  	\n{\Om(\ph)}_{m} + (1+\n{\Om(\ph)}_{1})^{4m-3}\n{\Om(\ph)}_{0} \bigr).
\]
\end{theorem}
The main feature of Theorem~\ref{b-est} is that the constants $c_{m}$ and $b_{m}$ can be chosen independently of $\ph$.

Note that the estimate (i) is linear in the highest Sobolev norm $\n{\ph}_{m}$ for $m\ge 2$, and that the estimate (ii) is linear in the highest weighted $h_{r}^{m}$-norm $\n{\Om(\ph)}_{m}$ of $\Om(\ph)$. Hence Theorem~\ref{b-est} shows that the 1-smoothing property of the Birkhoff map $\Om$ established in \cite{Kappeler:4WN-jiH9} holds in a uniform fashion.

The proof of Theorem~\ref{b-est} relies on estimates of the action variables $I(\ph) = (I_n)_{n\in\Z}$ where $I_{n} = (x_{n}^{2} + y_{n}^{2})/2$, $n\in\Z$. The decay properties of the actions $I_n$ are known to be closely related to the regularity properties of $\ph$ -- c.f. \cite{Kappeler:2009uk,Djakov:2006ba,Kappeler:2009kp}. We need to quantify this relationship by providing two sided estimates of the Sobolev norms of $\ph$ in terms of weighted $\ell^{1}$-norms of $I(\ph)$. For that purpose introduce the weighted sequence space $\ell^{1}_{s}$ whose norm is defined by
\[
  \n{I(\ph)}_{\ell^{1}_{s}} \defl \sum_{n\in\Z} \lin{2n\pi}^{s} \abs{I_n(\ph)}.
\]

\begin{theorem}
\label{act-sob-est}
For any integer $m\ge 1$, there exist absolute constants $c_m$ and $d_{m}$, such that for all $\ph\in H_{r}^{m}$
\[
  \emph{ (i)}\quad\n{I(\ph)}_{\ell^{1}_{2m}}
   \le 
  c_{m}^{2}\bigl(\n{\ph}_{m}^{2}
  +
  (1+\n{\ph}_{1})^{4m}\n{\ph}_{0}^{2}\bigr),
\]
\[
  \emph{(ii)}\quad\n{\ph}_m^{2} \le d_{m}^{2}\bigl(\n{I(\ph)}_{\ell^{1}_{2m}}
   + (1+\n{I(\ph)}_{\ell^{1}_{2}})^{4m-3}\n{I(\ph)}_{\ell^{1}}\bigr).
\]
\end{theorem}

\emph{Remark.} The same constants $c_{m}$, $d_{m}$ of Theorem~\ref{b-est} can be used in Theorem~\ref{act-sob-est}.

It turns out that versions of the estimates (i) of Theorems~\ref{b-est} \& \ref{act-sob-est} hold for a larger family of spaces referred to \emph{weighted Sobolev spaces} -- see \cite{Kappeler:1999er,Kappeler:2001hsa} for an introduction. A \emph{normalised, submultiplicative}, and \emph{monotone weight} is a symmetric function $w\colon\Z\to\R$  with
\[
  w_n \ge 1,\qquad w_{n} = w_{-n},\qquad w_{n+m}\le w_{n}w_{m},\qquad
  w_{\abs{n}}\le w_{\abs{n}+1},
\]
for all $n,m\in\Z$. The class of all such weights is denoted by $\Mc$ and $H_{c}^w$ is the space of $L^{2}_{c}$ functions $\ph$ with finite $w$-norm
\[
  \n{\ph}_w^2 \defl \sum_{n\in\Z} w_{2n}^2 (\abs{\ph_{2n}^{-}}^{2}+\abs{\ph_{2n}^{+}}^{2}).
\]
Furthermore, $h_{r}^{w}$ denotes the subspace of $\ell^{2}_{r}$ where 
$\n{x}_{w}^{2} + \n{y}_{w}^{2} < \infty$,
\[
  \n{x}_{w}^{2} \defl \sum_{n\in\Z} w_{2n}^{2} \abs{x_{n}}^{2}.
\]

For any $s \ge 0$, the \emph{Sobolev weight} $\lin{n\pi}^{s}$ gives rise to the usual Sobolev space $H^{s}_{c}$. For $s\ge 0$ and $a > 0$, the \emph{Abel weight} $\lin{n\pi}^s \e^{a\abs{n}}$ gives rise to the space $H^{s,a}_{c}$ of $L^2_{c}$-functions, which can be analytically extended to the open strip $\setdef{z}{\abs{\Im z} < a/2\pi}$ of the complex plane with traces in $H_{c}^s$ on the boundary lines. In between are, among others, the \emph{Gevrey weights}
\[
  \lin{n}^s \e^{a\abs{n}^\sigma},\qquad 0< \sigma < 1,\quad s\ge 0,\quad a > 0,
\]
which give rise to the Gevrey spaces $H^{s,a,\sigma}_{c}$, as well as weights of the form
\[
  \lin{n}^s\exp\biggl(\frac{a\abs{n}}{1+\log^{\sg}\lin{n}} \biggr),
    \qquad 0< \sigma < 1,\quad s\ge 0,\quad a > 0,
\]
that are lighter than Abel weights but heavier than Gevrey weights.

To avoid certain technicalities in our estimates, we restrict ourselves to weights incorporating a factor which grows at a linear rate. We thus introduce the subclass
\[
  \Mc^{1} \defl 
    \setdef*{w\in\Mc}
            {w_{n}=\lin{n}v_{n}\; \text{ for all }n\in\Z\text{ with some }v\in\Mc}.
\]
Finally, we assume all weights $w\in\Mc$ to be piecewise linearly extended to functions on the real line $w\colon \R\to\R_{>0}$, $t\mapsto w[t]$.

\begin{theorem}
\label{act-west}
For any weight $w\in\Mc^{1}$ there exists a complex neighbourhood $W_{w}$ of $H^{w}_{r}$ within $H^{w}_{c}$ and a constant $c_{w}$, such that
\[
  \emph{ (i)}\quad\sum_{n\in\Z} w_{2n}^{2}\abs{I_{n}}
   \le
  c_{w}^{2}w[16\n{\ph}_{w}^{2}]^{2}\n{\ph}_{w}^{2}.
\]
Moreover, the restriction of the Birkhoff map $\Om$ to $H^{w}_{r}$ takes values in $h_{r}^{w}$ and satisfies
\[
  \emph{ (ii)}\quad
  \n{\Om(\ph)}_{w} \le c_{w}w[16\n{\ph}_{w}^{2}]\n{\ph}_{w}.
\]
\end{theorem}

In this more general set up the bounds in (i) and (ii) of Theorem~\ref{act-west} are of the same type as the weight, and are valid for all submultiplicative weights including those growing exponentially fast. The following version of Theorem~\ref{act-west} for Sobolev spaces of real exponent complements the results of Theorems~\ref{act-sob-est}-\ref{act-west}.

\begin{corollary}
For any real $s\ge 1$ there exist a complex neighbourhood $W_{s}$ of $H^{s}_{r}$ and a constant $c_{s}$ such that
\[
  \n{I(\ph)}_{\ell^{1}_{2s}} \le c_{s}^{2}(1+\n{\ph}_{s})^{4s}\n{\ph}_{s}^{2}.
\]
Moreover, the restriction of the Birkhoff map $\Om$ to $H^{s}_{r}$ takes values in $h_{r}^{s}$ and satisfies
\[
  \n{\Om(\ph)}_{s} \le c_{s}(1+\n{\ph}_{s})^{2s}\n{\ph}_{s}.
\]
\end{corollary}

\emph{Update.} In this updated version of the article~\cite{Molnar:2014vg} we improve the estimates of Theorem~\ref{b-est} \& \ref{act-sob-est} so that the remainders depend only on the $H^{1}$-norm of $\ph$ and the $\ell_{2}^{1}$-norm of $I(\ph)$ instead of the $H^{m-1}$-norm and $\ell_{2m-2}^{1}$-norm, respectively.

\emph{Outline.} The action variables $I_{n}$ and more generally the action variables $J_{n,k}$ on levels $k\ge 1$ can be defined entirely in terms of the periodic spectrum of the associated Zakharov-Shabat operator used in the Lax-pair formulation of the \nls-equation. More to the point, consider the operator
\[
  L(\ph) = 
  \mat[\bigg]{\,\ii & \\  & -\ii} 
  \frac{\ddd}{\dx} +
  \mat[\bigg]{ & \psi \\ \ob\psi & },
\]
with periodic boundary conditions on the interval $[0,2]$ of twice the length of the periodicity of $\ph=(\psi,\ob\psi)\in L^{2}_{r}$. The spectrum of $L(\ph)$ is well known to be real, pure point, and to consist of a double infinite sequence of eigenvalues
\[
  \dotsb \le \lm_{n-1}^{+} < \lm_{n}^{-} \le \lm_{n}^{+} < \lm_{n+1}^{-} \le \dotsb
\]
with asymptotic behaviour $\lm_{n}^{\pm} \sim n\pi$ as $\abs{n}\to\infty$. 
The asymptotic behaviour of the actions on  odd levels $k=2m+1$ turns out to be
\[
  J_{n,2m+1} \sim (\lm_{n}^{\pm})^{2m}I_{n} \sim (n\pi)^{2m}I_{n},\qquad \abs{n}\to \infty,
\]
and they satisfy the trace formula
\[
  \sum_{n\ge 1} J_{n,2m+1} = \frac{(-1)^{m+1}}{4^{m}}\Hc_{2m+1},\qquad m\ge 1,
\]
where $\Hc_{2m+1}$ denotes the $2m+1$th Hamiltonian in the \nls-hierarchy,
\[
  \Hc_{1} = \int_{\T} \abs{\psi}^{2}\,\dx,\quad
  \Hc_{3} = \int_{\T} (\abs{\psi'}^{2} + \abs{\psi}^{4})\,\dx,\quad \dotsc.
\]
Note that $\Hc_{3}$ denotes the Hamiltonian of the \nls equation. On $H_{r}^{m}$, for $m\ge 1$,
\[
  \Hc_{2m+1} = (-1)^{m+1}\int_{\T} \left(\abs{\ps^{(m)}}^{2} + p_{m}(\ps,\dotsc,\ps^{(m-1)})\right)\,\dx,
\]
where $p_{m}$ is a polynomial expression in $\psi$ and its first $m-1$ derivatives.

Viewing $\Hc_{2m+1}$ as a lower order perturbation of the $H_{r}^{m}$-norm we formally obtain at first order
\[
  \sum_{n\in\Z} (n\pi)^{2m} I_{n}
   \sim \sum_{n\in\Z} J_{n,2m+1}
   \sim \n{\ph}_{m}^{2}.
\]

The essential ingredient to make this formal statement explicit is a sufficiently accurate localisation of the periodic eigenvalues $\lm_{n}^{\pm}$ whose threshold in $\abs{n}$ depends only on $\n{\ph}_{1}$. For $\abs{n}$ above this threshold we can directly compare the weighted action norms and the polynomial expressions in $\ph$ as described above, while the remainder for $\abs{n}$ below the threshold can be regard as a $H^{1}$-error term. Thereof we obtain Theorem~\ref{act-sob-est}, which directly implies Theorem~\ref{b-est}. Note that our method of proof completely avoids the use of auxiliary spectral quantities such as \emph{spectral heights} or \emph{conformal mapping theory}. 

To prove Theorem~\ref{act-west}, we take a slightly different approach by quantitatively estimating the action variables in terms of the spacing of the periodic eigenvalues of the associated Zakharov-Shabat operator. For the latter we obtain estimates in any weighted norm, which allows us to obtain Theorem~\ref{act-west}.

%

\emph{Related results.}
Theorem~\ref{b-est} for $m=0$ is referred to as Parseval's identity,
\[
  \frac{1}{2}\n{\Om(\ph)}_{0} = \n{I(\ph)}_{\ell^{1}} = \frac{1}{2} \n{\ph}_{0}^{2},
\]
and is well known -- see \cite{Grebert:2014iq,McKean:1997ka}. The case $m=1$ was proved by Korotyaev~\cite{Korotyaev:2005fb} using conformal mapping theory, see also \cite{Korotyaev:2010ft}. However, his method does not seem applicable for the case $m\ge 2$. In fact, it is stated as an open problem in \cite{Korotyaev:2010ft}.

For the case of the KdV-equation
\[
  \partial_{t}u = -\partial_{x}^{3}u + 6u\partial_{x}u,\qquad x\in\T,
\]
Korotyaev~\cite{Korotyaev:2000tc,Korotyaev:2006uh} obtained polynomial bounds of the Sobolev norms $\n{u}_{m}$ in terms of the action variables where the order of the polynomials grows factorial in $m$. Note that the bound in (ii) of Theorem~\ref{act-sob-est} is of order 1 in the Sobolev norm $\n{\ph}_{m}$ and the order of the remainder grows linearly in $m$. It turns out that our method can also be applied to the \kdv-equation. In \cite{Molnar:_ROURXz4} we improve on the bounds obtained by Korotyaev in \cite{Korotyaev:2000tc,Korotyaev:2006uh}.

For \nls in weighted Sobolev spaces the qualitative relationship
\[
  \ph\in H_{r}^{w}\quad \iff \quad \Om(\ph) \in h_{r}^{w}
\]
is a corollary of the methods presented in \cite{Kappeler:2009kp,Djakov:2006ba} -- at least for weights growing at subexponential speed. To the best of our knowledge, the estimate of $\n{\Om(\ph)}_{w}$ on $H_{r}^{w}$ as well as the estimate of $\n{I(\ph)}_{w}$ on a small complex neighbourhood of $H_{r}^{w}$ as presented in Theorem~\ref{act-west} are new.

Viewing the action $I_{n}$ as a 1-smoothing perturbation of the squared modulus of the $n$th Fourier coefficient, our method of comparing the weighted action norms with the Hamiltonians of the \nls-hierarchy amounts to a separate analysis of Fourier modes of low and high frequencies. This idea has a long history in the analysis of nonlinear PDEs. Most recently, it lead Colliander, Keel, Staffilani, Takaoka \& Tao \cite{Colliander:2001wg,Colliander:2001fc,Colliander:2004gc,Colliander:2010fs} to invent the I-Method, which allows to obtain global well posedness of subcritical equations in low regularity regimes where the Hamiltonian (or other integrals) of the equation cease to be well defined. The idea is to damp all sufficiently high Fourier modes of a local solution such that the Hamiltonian can be controlled by weaker norms while still being an >>almost conserved<< quantity. The difficulty here is to choose the damping subtle enough such that the nonlinearity of the equation does not create a significant interaction of low and and high frequencies. Our situation is so to say opposed to that of the I-Method: As we aim for quantitative global estimates, controlling the modes of low frequencies is the most delicate part. 
%
%
Here the localisation of the periodic eigenvalues of the Zakharov-Shabat operator associated with the \nls equation plays a crucial role. Note that there exists a vast amount of literature on the spectral theory of these operators -- see e.g. \cite{Korotyaev:2010ft,Grebert:1998cz,Kappeler:2001hsa,Djakov:2006ba,Grebert:2014iq} and the references therein.

\emph{Organisation of the paper.} In section~\ref{s:setup} we recall the definition of the \nls action variables on integer levels $k\ge 1$ as well as the trace formulae relating them to the hierarchy of \nls Hamiltonians. The somewhat lengthy proof of the analyticity properties of the action integrand can be found in the appendix~\ref{a:ana}.  The quadratic localisation of the Zakharov-Shabat spectrum is obtained in section~\ref{s:trap-sp}, and is subsequently used in sections~\ref{s:act-est} and \ref{s:sob-est} to obtain Theorem~\ref{act-sob-est} (i) and (ii), respectively. In the final section~\ref{s:act-west} we obtain the estimate of the actions in terms of the spacing of the Zakharov-Shabat eigenvalues which implies Theorem~\ref{act-west}.

\emph{Acknowledgement.} The author is very grateful to Professor Thomas Kappeler for continued support and helpful comments on this paper.

\section{Setup}
\label{s:setup}

In this section we briefly recall the definition of \nls action variables from~\cite{Grebert:2014iq}, as well as the main properties of the periodic spectrum of Zakharov-Shabat operators used to define them.

For a \emph{potential} $\ph=(\phm,\php)\in H_{c}^0 = L^2_c$ consider the Zakharov-Shabat operator
\[
  L(\ph) \defl
  \mat[\bigg]{\,\ii & \\  & -\ii} 
  \frac{\ddd}{\dx} +
  \mat[\bigg]{ & \phm \\ \php & }
\]
on the interval $[0,2]$ with periodic boundary conditions. The spectrum of $L(\ph)$ is well known to be pure point, and more precisely, to consist of a sequence of pairs of complex eigenvalues $\lm_n^+(\ph)$ and $\lm_n^-(\ph)$, listed with algebraic multiplicities, such that
\[
  \lm_n^\pm(\ph) = n\pi + \ell^2_n,\qquad n\in\Z.
\]
Here $\ell_n^2$ denotes a generic $\ell^2$-sequence. We may order the eigenvalues lexicographically -- first by their real part, and second by their imaginary part -- to represented them as a bi-infinite sequence of complex eigenvalues
\[
  \dotsb \lex \lambda_{n-1}^+ \lex \lambda_{n}^- \lex \lambda_{n}^+ \lex \lambda_{n+1}^- \lex \dotsb.
\]
By a slight abuse of notation, we call the eigenvalues of $L(\ph)$ the \emph{periodic spectrum of $\ph$}. Further we introduce the \emph{gap lengths}
\[
  \gm_n \defl \lm_n^+ - \lm_n^- = \ell^2_n,\qquad n\in\Z.
\]

To obtain another characterisation of the periodic spectrum, we denote by $M(x,\lm,\ph)$ the standard fundamental solution of the ordinary differential equation $L(\ph)M = \lm M$, and introduce the discriminant
\[
  \Dl(\lm,\ph) \defl  \operatorname{tr} M(1,\lm,\ph).
\]
To simplify matters, we may drop some or all of its arguments from the notation, whenever there is no danger of confusion. The periodic spectrum of $\ph$ is precisely the zero set of the entire function $\Dl^2(\lm) - 4$, and we have the product representation
\[
  \Dl^2(\lm) - 4
   = 
  -4\prod_{n\in\Z}
  \frac{(\lm_n^+ - \lm)(\lm_n^- - \lm)}{\pi_n^2},
  \quad
  \pi_n
   \defl 
  \begin{cases}
  n\pi, & n\neq 0,\\
  1, & n=0.
  \end{cases}
\]
Hence, this function is a spectral invariant. We also need the $\lm$-derivative $\dDl\defl\partial_{\lm}\Dl$ whose zeros are denoted by $\lm_{n}^{\ld}$ and satisfy $\lm_{n}^{\ld} = n\pi + \ell^{2}_{n}$. This derivative has the product representation
\[
  \dDl(\lm) = 2\prod_{n\in\Z} \frac{\lm_{n}^{\ld}-\lm}{\pi_{n}}.
\]

If the potential $\ph$ is of real type, as in the case of the defocusing \nls equation, the periodic spectrum is real, the eigenvalues are characterised by Floquet theory, and the lexicographical ordering reduces to the ordering of real numbers
\[
  \dotsb
   \le \lm_{n-1}^{+} < \lm_{n}^{-} \le \lm_{n}^{\ld} \le \lm_{n}^{+} < \lm_{n+1}^{-}
   \le \dotsb.
\]
Each $\ph\in L^{2}_{r}$ has an open neighbourhood $V_{\ph}$ within $L^{2}_{c}$ for which there exist disjoint closed discs $(U_{n})_{n\in\Z}$ centred on the real axis with the properties:

\begin{itemize}
\item[(i)] $\lm_{n}^{\pm}(\psi)$ and $\lm_{n}^{\ld}(\psi)$ are contained in the interior of $U_{n}$ for every $\psi\in V_{\ph}$,

\item[(ii)] there exists a constant $c \ge 1$ such that for $m\neq n$,
\[
  c^{-1}\abs{m-n} \le \dist(U_{n},U_{m}) \le c\abs{m-n},
\]

\item[(iii)] $U_{n} = \setd{\abs{\lm-n\pi} \le \pi/4}$ for $\abs{n}$ sufficiently large.
\end{itemize}

\noindent
Such discs are called \emph{isolating neighbourhoods}. The union of all $V_{\ph}$ defines an open and connected neighbourhood of $L^{2}_{r}$ within $L^{2}_{c}$ and is denoted by $W$.
Throughout this text $V_{\ph}$ denotes a neighbourhood of $\ph$ such that a common set of isolating neighbourhoods for all $\psi\in V_{\ph}$ exists.

Following Flaschka \& McLaughlin's approach for the \kdv-equation \cite{Flaschka:1976tc}, one can define action variables for the defocusing \nls-equation by Arnold's formula -- see also \cite{McKean:1997ka}
\[
  I_n
   = 
  \frac{1}{\pi}\int_{a_n}
  \frac{\lm \dDl(\lm)}{\sqrt{\Dl^2(\lm)-4}}
  \,\dlm.
\]
Here $a_{n}$ denotes a cycle around $(\lm_{n}^{-},\lm_{n}^{+})$ on the spectral curve
\[
  C_{\ph} = \setdef{(\lm,z)}{z^2 = \Dl^2(\lm,\ph) - 4}\subset\C,
\]
on which the square root $\sqrt{\Dl^2(\lm)-4}$ is defined. This curve is another spectral invariant associated with $\ph$, and an open Riemann surface of infinite genus if and only if the periodic spectrum of $\ph$ is simple. To avoid the technicalities involved with this curve, we follow the approach presented in \cite{Grebert:2014iq} and fix proper branches of the square root which allows us to reduce the definition of the actions to standard contour integrals in the complex plane.

Firstly, we denote by $\sqrt[+]{\phantom{\lm}}$ the \emph{principal branch} of the square root on the complex plane minus the ray $(-\infty,0]$. Secondly, we define the \emph{standard root}
\[
  \vs_{n}(\lm) = \sqrt[\mathrm{s}]{(\lm_{n}^{+}-\lm)(\lm_{n}^{-}-\lm)},
                     \qquad \lm\notin [\lm_{n}^{-},\lm_{n}^{+}],
\]
by the condition
\begin{align}
  \label{s-root}
  \vs_{n}(\lm) \defl (\tau_{n}-\lm)\sqrt[+]{1 - \gm_{n}^{2}/4(\tau_{n}-\lm)},
  						 \qquad \tau_{n} = (\lm_{n}^{-}+\lm_{n}^{+})/2,
\end{align}
for all $\abs{\lm}$ sufficiently large. For any $\ph\in W$ the standard root is analytic in $\lm$ on $\C\setminus[\lm_{n}^{-},\lm_{n}^{+}]$ and in $(\lm,\psi)$ on $(\C\setminus U_{n})\times V_{\ph}$. Thirdly, we define the \emph{canonical root} $\sqrt[c]{\Dl^{2}(\lm)-4}$ by the product representation
\[
  \sqrt[c]{\Dl^{2}(\lm)-4} \defl 2\ii\prod_{m\in\Z} \frac{\vs_{m}(\lm)}{\pi_{m}}.
\]
For any $\ph\in W$ this root is analytic in $\lm$ on $\C\setminus\bigcup_{\gm_{n}\neq 0} [\lm_{n}^{-},\lm_{n}^{+}]$ and in $(\lm,\psi)$ on $(\C\setminus \bigcup_{n\in\Z} U_{n})\times V_{\ph}$ -- see \cite[Section 12]{Grebert:2014iq} for all the details.

The $n$th \nls action variable of $\ph\in W$ is then given by
\[
  \quad I_n
   \defl 
  \frac{1}{\pi}\int_{\Gm_n}
  \frac{\lm \dDl(\lm)}{\sqrt[c]{\Dl^2(\lm)-4}}\,\dlm,
\]
where $\Gm_{n}$ denotes any sufficiently close circuit around $[\lm_{n}^{-},\lm_{n}^{+}]$. More generally, the $n$th action on level $k\ge 1$ is given by
\[
  J_{n,k}
   \defl
  \frac{1}{k\pi}\int_{\Gm_n}
  \frac{\lm^{k}\dDl(\lm)}{\sqrt[c]{\Dl^{2}(\lm)-4}}\,\dlm.
\]
It was shown in \cite{Grebert:2014iq} and, for convenience of the reader, will be reproved in the sequel that each action variable is analytic on $W$ and vanishes if and only if $\gm_{n}$ is zero.

If $\ph$ is of real type, then all actions are real and those on odd levels, such as $J_{n,1} = I_{n}$, are nonnegative. Moreover, the actions on level one are well known to satisfy the trace formula,
\begin{align}
  \label{tf-1}
  \sum_{n\in\Z} I_{n}(\ph) = \Hc_{1}(\ph) = \frac{1}{2}\n{\ph}_{0}^{2}
   = \frac{1}{2}\int_{\T} (\abs{\phm}^{2} + \abs{\php}^{2})\,\dx.
\end{align}
Similar trace formulae have been derived by McKean \& Vaninsky~\cite{McKean:1997ka} expressing the actions on any level $k\ge 1$ in terms of Hamiltonians of the \emph{\nls-hierarchy}. The first three Hamiltonians of this hierarchy are
\begin{align*}
  \Hc_{1}(\ph) &= \phantom{\frac{1}{2}} \int_{\T} \phm\php\,\dx,\\
  \Hc_{2}(\ph) &= \frac{1}{2} \int_{\T} (\phm'\php - \phm\php')\,\dx,\\
  \Hc_{3}(\ph) &= \phantom{\frac{1}{2}} \int_{\T} (\phm'\php' + \php^{2}\phm^{2})\,\dx,
\end{align*}
and in general, for any sufficiently regular $\ph\in L^{2}_{c}$,
\[
  \Hc_{k+1}(\ph) = 
  \int_{\T} (-\phm^{\phantom{.}}\php^{(k)} + q_{k}(\ph,\dotsc,\ph^{(k-1)}))\,\dx,\qquad k\ge 1,
\]
with $q_{k}$ being a canonically determined polynomial in $\ph$ and its first $k-1$ derivatives -- see appendix~\ref{s:hamil}. The following version of the trace formula is taken from \cite{Grebert:2014iq}.

\begin{theorem}[Trace Formula]
\label{tf}
For any $k\ge 2$ and any $\ph\in H_{c}^{k-1}\cap W$,
\begin{align}
  \label{tf-k}
  \sum_{n\in\Z} J_{n,k}(\ph) = -\frac{1}{(2\ii)^{k-1}}\Hc_{k}(\ph).
\end{align}
\end{theorem}

In particular, for every sufficiently regular real type potential
\[
  \sum_{n\in\Z} J_{n,2m+1}(\ph) =
   \frac{1}{4^{m}} \int_{\T} \bigl(\abs{\ph^{(m)}}^{2} + \dotsb\bigr)\,\dx,\qquad m\ge 0.
\]
This identity is used in sections~\ref{s:act-est} and \ref{s:sob-est} to estimate the actions on level $2m+1$ in terms of $\n{\ph}_{m}$. In order to obtain thereof estimates for the action variables on level one, a detailed analysis of the analytical properties of the action integrand is necessary. To this end, we define for any $\ph\in W$ on $(\C\setminus \bigcup_{n\in\Z} U_{n}) \times V_{\ph}$ the complex 1-form
\begin{align}
  \label{om}
  \om(\lm,\psi) 
  \defl \frac{\dDl(\lm,\psi)}{\sqrt[c]{\Dl^{2}(\lm,\psi)-4}}\,\dlm
      = -\ii\prod_{m\in\Z} \frac{\lm_{m}^{\ld}(\psi)-\lm}{\vs_{m}(\lm,\psi)} \dlm.
\end{align}
We call a path in the complex plane \emph{admissible} for $\ph\in L^{2}_{c}$ if, except possibly at its endpoints, it does not intersect any gap $[\lm_{n}^{-}(\ph),\lm_{n}^{+}(\ph)]$.

\begin{lemma}
\label{w-closed}
For each $\ph\in W$, the 1-form $\om$ has the following properties:
\begin{itemize}
\item[(i)]
$\om$ is analytic on $(\C\setminus \bigcup_{n\in\Z} U_{n}) \times V_{\ph}$, 
\item[(ii)]
$\om(\cdot,\ph)$ is analytic on $\C\setminus \bigcup_{\gm_{n}\neq 0} [\lm_{n}^{-},\lm_{n}^{+}]$, and
\item[(iii)]
for any admissible path from $\lm_{n}^{-}$ to $\lm_{n}^{+}$ in $U_{n}$,
\[
  \int_{\lm_{n}^{-}}^{\lm_{n}^{+}} \om =  0.
\]
In particular, for any closed circuit $\Gm_{n}$ in $U_{n}$ around $[\lm_{n}^{-},\lm_{n}^{+}]$,
\[
  \int_{\Gm_{n}} \om = 0.
\]
\end{itemize}
\end{lemma}

\begin{proof}
Since $\Dl^{2}(\lm,\psi)-4$ vanishes if and only if $\lm$ is a periodic eigenvalue, and numerator and dominator of $\om(\lm,\psi)$ are analytic on $(\C\setminus \bigcup_{n\in\Z} U_{n}) \times V_{\ph}$, the first claim follows immediately.

By the same reasoning, $\om(\cdot,\ph)$ is analytic on $\C\setminus \bigcup_{n\in\Z} [\lm_{n}^{-},\lm_{n}^{+}]$. Moreover, if the $n$th gap is collapsed, that is $\lm_{n}^{+} = \lm_{n}^{-} \defr \lm_{n}$, then $\Dl^{2}(\lm)-4$ has a double root at $\lm_{n}$, and hence
\[
  0 = \Dl(\lm_{n})\dDl(\lm_{n}) = 2(-1)^{n}\dDl(\lm_{n}).
\]
As $\dDl(\lm)$ has a single root in $U_{n}$, namely $\lm_{n}^{\ld}$, we conclude $\lm_{n}^{\ld} = \lm_{n} = \lm_{n}^{\pm}$, and the $n$th term in the product representations of $\dDl$ and $\sqrt[c]{\Dl^{2}-4}$ cancels. So $\om(\cdot,\ph)$ is analytic on all of $U_{n}$, and the second claim follows.

We first consider the case where $\ph$ is of real type. Then $(-1)^{n}\Dl(\lm,\ph) \ge 2$ on $[\lm_{n}^{-},\lm_{n}^{+}]$, as shown in \cite{Grebert:2014iq}, and hence
\[
  \min_{\lm_{n}^{-} \le \lm\le \lm_{n}^{+}} (-1)^{n}\Dl(\lm,\ph) - \sqrt[+]{\Dl^{2}(\lm,\ph)-4} > 0.
\]
Thus, if we choose a circuit $\Gm$ sufficiently close to $[\lm_{n}^{-},\lm_{n}^{+}]$, and a sufficiently small neighbourhood $V\subset V_{\ph}$ of $\ph$, then $[\lm_{n}^{-}(\psi),\lm_{n}^{+}(\psi)]$ is enclosed by $\Gm$ and the real part of $(-1)^{n}(\Dl(\lm,\psi) - \sqrt[c]{\Dl^{2}(\lm,\psi)-4})$ is positive on that circuit for all $\psi\in V$. In consequence, the principal branch of the logarithm
\[
  l_{n}(\lm,\psi) = \log\frac{(-1)^{n}}{2}\left(\Dl(\lm,\psi) + \sqrt[c]{\Dl^{2}(\lm,\psi)-4}\right)
\]
is analytic in a neighbourhood of $\Gm$ and $\ddd l_{n} = \om$. It follows that the analytic functional $\psi \mapsto \int_{\partial U_{n}} \om$ vanishes on an open neighbourhood of $\ph$, and hence on all of $V_{\ph}$. Since $W = \bigcup_{\ph\in L^{2}_{r}} V_{\ph}$, it follows that $\int_{\Gm_{n}}\om = 0$ for every $\ph\in W$.

Finally, consider $\int_{\lm_{n}^{-}}^{\lm_{n}^{+}} \om$ with the path of integration chosen to be admissible. As $\om$ is closed around the gap, the integral does not depend on the chosen admissible path. Suppose $\gm_{n} \neq 0$, then by the product representation \eqref{om} of $\om$,
\begin{align}
  \label{chi-om}
  \int_{\lm_{n}^{-}}^{\lm_{n}^{+}} \om
   = -\ii \int_{\lm_{n}^{-}}^{\lm_{n}^{+}} \frac{\lm_{n}^{\ld}-\lm}{\vs_{n}(\lm)}\chi_{n}(\lm)\,\dlm,
   \qquad \chi_{n}(\lm) \defl \prod_{m\neq n} \frac{\lm_{m}^{\ld} - \lm}{\vs_{m}(\lm)},
\end{align}
where $\chi_{n}(\lm)$ is analytic on $\C\setminus\bigcup_{m\neq n}[\lm_{m}^{-},\lm_{m}^{+}]$ -- see \cite[Section 12]{Grebert:2014iq}. Further, by the definition of the standard root \eqref{s-root}, the function
\[
  z \mapsto \vs_{n}(\tau_{n} + z \gm_{n}/2) = -z\sqrt[+]{1-z^{-2}}\gm_{n}/2
\]
is analytic on $\C\setminus[-1,1]$. For $s \in (-1,1)$ consider $z_{s} = s \pm \ii \ep$ to conclude
\[
  \vs_{n}(\tau_{n} + (s \pm \ii 0) \gm_{n}/2) = \mp \ii \sqrt[+]{1-s^{2}}\gm_{n}/2.
\]
In particular, if $\al_{1}$ and $\al_{2}$ are admissible paths from $\lm_{n}^{-}$ to $\lm_{n}^{+}$ running on different sides of $[\lm_{n}^{-},\lm_{n}^{+}]$, then the integrand takes the opposite sign on these paths and hence
\[
  \int_{\al_{1}} \om = -\int_{\al_{2}} \om.
\]
On the other hand, as $\int_{\partial U_{n}} \om = 0$, the integral is independent of the chosen path $\al_{i}$ and thus needs to be zero.\qed
\end{proof}

If we write the action variables as
\begin{align}
  \label{Jn-om}
  J_{n,k} = \frac{1}{k\pi}\int_{\Gm_{n}} \lm^{k}\om,
\end{align}
then the analyticity on $W$ is evident. To proceed we need to find a globally defined primitive of $\om$. So for $\ph\in W$ we define on $(\C\setminus\bigcup_{n\in\Z} U_{n}) \times V_{\ph}$ the mapping
\[
  F(\lm,\psi) \defl \frac{1}{2}\left(\int_{\lm_{0}^{-}(\psi)}^{\lm} \om(\mu,\psi) + \int_{\lm_{0}^{+}(\psi)}^{\lm} \om(\mu,\psi)\right),
\]
where the paths of integration are chosen to be admissible. These improper integrals exist, as for $\gm_{0} = 0$ the integrand  is analytic on $U_{0}$, while for $\gm_{0}\neq 0$ it is of the form $1/\sqrt{\lm_{0}^{\pm}-\lm}$ locally around $\lm_{0}^{\pm}$. By Lemma~\ref{w-closed} they are also independent of the chosen admissible path. Hence $F$ is well defined and one has
\[
  F(\lm,\psi) = \int_{\lm_{0}^{-}(\psi)}^{\lm} \om(\mu,\psi) = \int_{\lm_{0}^{+}(\psi)}^{\lm} \om(\mu,\psi).
\]

Even though the eigenvalues $\lm_{0}^{\pm}$ are, due to their lexicographical ordering, not even continuous on $W$, the mapping $F$ turns out to be differentiable.

\begin{lemma}
\label{F-prop}
For every $\ph\in W$, we have that

\begin{itemize}
\item[(i)] $F$ is analytic on $(\C\setminus\bigcup_{n\in\Z} U_{n}) \times V_{\ph}$, and

\item[(ii)] $F(\cdot,\ph)$ can be continuously extended onto $\C\setminus \bigcup_{\gm_{n}\neq 0} (\lm_{n}^{-}, \lm_{n}^{+})$ with
\[
  F(\lm_{n}^{+},\ph) = F(\lm_{n}^{-},\ph),\qquad n\in\Z.
\]
\item[(iii)] If, in addition, $\ph\in L^{2}_{r}$, then locally around $[\lm_{n}^{-},\lm_{n}^{+}]$
\[
  F(\lm) = l_{n}(\lm) - \ii n\pi,\qquad l_{n}(\lm) = \log\frac{(-1)^{n}}{2}\left(\Dl(\lm) + \sqrt[c]{\Dl^{2}(\lm)-4}\right).
\]
In particular, for any real $\lm_{n}^{-} < \lm < \lm_{n}^{+}$,
\[
  F(\lm \pm \ii 0) = \pm f_{n}(\lm) - \ii n\pi,
  \qquad f_{n}(\lm) = \cosh^{-1}((-1)^{n}\Dl(\lm)/2).
\]
Clearly, $f_{n}$ is continuous on $[\lm_{n}^{-},\lm_{n}^{+}]$, strictly positive on $(\lm_{n}^{-},\lm_{n}^{+})$, and vanishes at the boundary points.
\item[(iv)] At the zero potential one has $F(\lm,0) = -\ii \lm$.

\end{itemize}

\end{lemma}

\begin{proof}
The proof of (i) is standard but a bit technical and can be found in appendix~\ref{a:ana}, and claim (ii) follows immediately from the properties of $\om$.

(iii): Note that locally around $[\lm_{n}^{-},\lm_{n}^{+}]$ both $l_{n}$ and $F$ are primitives of $\om$ and hence are identical up to a additive constant which may depend on $\ph$. Clearly, $l_{n}(\lm_{n}^{\pm}) = 0$. On the other hand, since $\int_{\lm_{k}^{-}}^{\lm_{k}^{+}}\om = 0$ for any $k$,
\[
  F(\lm_{n}^{\pm}) = \sum_{k=0}^{n-1} \int_{\lm_{k}^{+}}^{\lm_{k+1}^{-}} \om,\qquad
  F(\lm_{-n}^{\pm}) = \sum_{k=0}^{n-1} \int_{\lm_{-k}^{-}}^{\lm_{-k-1}^{+}} \om,\qquad n > 0.
\]
As $\ii(-1)^{k}\sqrt[c]{\Dl^{2}(\lm)-4} > 0$ for $\lm_{k}^{+} < \lm < \lm_{k+1}^{-}$ -- see \cite[Section 12]{Grebert:2014iq} -- we find
\[
  \int_{\lm_{k}^+}^{\lm_{k+1}^-} \om = 
  \ii(-1)^{k}\int_{\lm_{k}^+}^{\lm_{k+1}^-}
  \frac{\dot\Dl(\lm)}{\sqrt[+]{4-\Dl^2(\lm)}}\,\dlm
   =
  \ii(-1)^{k}
  \arcsin\frac{\Dl(\lm)}{2}\bigg|_{\lm_{k}^+}^{\lm_{k+1}^-}
   =
  -\ii \pi.
\]
Consequently, $F(\lm_{n}^{\pm})  = -\ii n\pi$ and $F-l_{n} \equiv -\ii n\pi$. Finally, $\pm (-1)^{n}\sqrt[c]{\Dl^{2}(\lm\pm \ii 0)-4} > 0$ for $\lm_{n}^{-} < \lm < \lm_{n}^{+}$ hence
\begin{align*}
  F(\lm \pm \ii 0)
  &= \log\frac{1}{2}\left((-1)^{n}\Dl(\lm) \pm \sqrt[+]{\Dl^{2}(\lm)-4}\right) - \ii n\pi\\
  &= \pm f_{n}(\Dl(\lm)/2) - \ii n\pi.
\end{align*}

(iv): At the zero potential, $\om(\lm,0) = -\ii\dlm$ which is evident from the product representation \eqref{om} of $\om$.\qed
\end{proof}

Given $\ph\in W$ we can integrate by parts in \eqref{Jn-om} to obtain
\[
  J_{n,k} =  \frac{1}{k\pi}\int_{\Gm_n}\lm^{k}\om
          = -\frac{1}{\pi}\int_{\Gm_n} \lm^{k-1}F(\lm)\,\dlm.
\]
Clearly, $J_{n,k}$ vanishes if $\gm_{n}=0$ in view of the analyticity of the integrand. Further, provided $\ph$ is of real type, then we may shrink the contour of integration to $[\lm_{n}^{-},\lm_{n}^{+}]$, and use the properties of $F$ to the effect that
\begin{align}
  \label{Jn-fn}
  J_{n,k} =  \frac{2}{\pi} \int_{\lm_{n}^{-}}^{\lm_{n}^{+}} \lm^{k-1} f_{n}(\lm) \,\dlm.
\end{align}
Thus on $L^{2}_{r}$ all actions are real and those on odd levels are nonnegative. Moreover, by the mean value theorem,
\begin{align}
  \label{zt-In-Jn}
  J_{n,2m+1} = \zt_{n,m}^{2m} I_{n},
\end{align}
for some $\zt_{n,m}\in [\lm_{n}^{-},\lm_{n}^{+}]$. Recall that $\lm_{n}^{\pm}\sim n\pi$ hence for any $m\ge 0$ we have $\zt_{n,m}^{2m} \sim (n\pi)^{2m}$ asymptotically as $\abs{n}\to\infty$. A quantitative estimate of the high level actions $J_{n,2m+1}$ in terms of the actions $I_{n}$ will be obtained in the next section.

Finally, consider the case of a potential with only finitely many open gaps. Then $F$ is analytic outside some sufficiently large circle and thus admits a Laurent expansion around zero. The coefficients of this expansion turn out to be the Hamiltonians of the \nls-hierarchy.

\begin{lemma}
\label{F-asy}
For any finite gap potential $\ph\in  W$ there exists $\Lm > 0$ such that
\[
  F(\lm,\ph)
   = -\ii \lm - \frac{\Hc_{1}(\ph)}{2\ii \lm} + \sum_{n\ge 2} \frac{\Hc_{n}(\ph)}{(2\ii\lm)^{n}},
   \qquad \abs{\lm} > \Lm.
\]
\end{lemma}

\begin{proof}
We attribute the claim to the asymptotic expansion of $\cosh^{-1}(\Dl/2)$ along the real axis -- see \cite{Grebert:2014iq,Kappeler:4WN-jiH9}. By the basic estimates for the discriminant
\[
  \Dl(\ii \tau,\ph) = 2\cosh\tau + o(\e^{\tau}),\qquad \tau\to\infty.
\]
Hence on an open neighbourhood $U$ of $\ii[\tau_{0},\infty)$ for $\tau_{0} > 0$ sufficiently large $\cosh^{-1}(\Dl(\lm,\ph)/2)$ is well defined. Here $\cosh^{-1}$ denotes the principal branch of the inverse of $\cosh$ which is defined on $\C\setminus(-\infty,1)$ and real valued on $[1,\infty)$. On this neighbourhood $U$, the $\lm$-derivatives of $\cosh^{-1}(\Dl/2)$ and $F$ coincide except for possibly the sign of
\[
  \Re \sqrt[c]{\Dl^{2}(\ii\tau)-4},\qquad \tau\ge \tau_{0}.
\]
This sign is locally constant in $\ph$ provided $\tau\ge \tau_{0}$ and, as the straight line $\setdef{t\ph}{0\le t\le 1}$ is compact in $L^{2}_{c}$, it can be determined by deforming $\ph$ to the zero potential. With
\[
  \sqrt[c]{\Dl^{2}(\ii\tau,0)-4} = 2\ii\prod_{m\in\Z} \frac{\vs_{m}(\ii\tau,0)}{\pi_{m}}
   = 2\ii\prod_{m\in\Z} \frac{m\pi - \ii\tau}{\pi_{m}} = 2\sinh \tau,
\]
we conclude the sign is positive on $U$, and consequently
\[
  \cosh^{-1}\frac{\Dl(\lm,\ph)}{2} = F(\lm,\ph) + c(\ph),\qquad \lm\in U,
\]
with an analytic function $c\colon  W\to \C$. For $\ph\in W$ fixed, the right hand side is analytic on $\C\setminus\bigcup_{\gm_{n}\neq 0}[\lm_{n}^{-},\lm_{n}^{+}]$ and continuous on $\C\setminus\bigcup_{\gm_{n}\neq 0}(\lm_{n}^{-},\lm_{n}^{+})$, hence the left hand side extends uniquely from $U$ onto the same domain. Moreover, $F$ vanishes at $\lm_{0}^{\pm}$, while
\[
  \cosh^{-1}\frac{\Dl(\lm_{0}^{\pm},\ph)}{2} \in \setdef{\ii m\pi}{m\in\Z}.
\]
Thus $c(\ph) = \ii m\pi$, and since $c$ is continuous, this $m$ is fixed on all of $ W$. Evaluating at the zero potential reveals $\cosh^{-1}(\Dl(\lm,0)/2) = -\ii \lm = F(\lm,0)$ hence $c \equiv 0$.

If $\ph$ is a finite gap potential, then $\cosh^{-1}(\Dl/2) = F$ is analytic outside a sufficiently large circle enclosing all open gaps, and has the following asymptotic expansion along the real axis \cite{Kappeler:4WN-jiH9}
\[
  \cosh^{-1}\frac{\Dl(\lm,\ph)}{2} = -\ii \lm - \frac{\Hc_{1}(\ph)}{2\ii \lm} + \sum_{n\ge 2} \frac{\Hc_{n}(\ph)}{(2\ii\lm)^{n}},
   \qquad \lm \to \pm\infty.
\]
This proves the claim.\qed
\end{proof}

\section{Localising the Zakharov-Shabat Spectrum}
\label{s:trap-sp}

The goal for this section is to provide a sufficiently accurate localisation of the spectrum of the Zakharov-Shabat operator
\[
  L(\ph) =
  \mat[\bigg]{\,\ii & \\  & -\ii\,} 
  \frac{\ddd}{\dx} +
  \mat[\bigg]{ & \phm \\ \php & },
\]
allowing to quantify the asymptotic relation $J_{n,2m+1} \sim (n\pi)^{2m} I_{n}$. Since one can translate the spectrum of $\ph$ without changing the $L^2$-norm, one can not obtain a uniform localisation on bounded subsets of $L^{2}_{c}$. Instead, we have to impose some regularity on $\ph$.

\begin{theorem}
\label{ev-trap}
Suppose $\ph\in H_c^1$, then for each $\lin{n}\ge 8\n{\ph}_1^2$,
\[
  \abs{\lm_{n}^{\pm}-n\pi} \le
  \frac{\n{\ph}_{1}^{2}}{\lin{n}}
   + \frac{\sqrt{2}\n{\ph}_{1}}{\lin{2n}} \le \pi/5,
\]
while the remaining eigenvalues are contained in the box
\[
  \setdef*{\lm\in\C}{\abs{\Re \lm} \le (8\n{\ph}_1^2-1/2)\pi,\quad
  \abs{\Im \lm} \le \n{\ph}_{1}}.
\]
\end{theorem}

\emph{Remark.}
In \cite{Li:1994td} Li \& McLaughlin obtained a localisation for $\ph$ in $H^{1}_{c}$ where the bound on the threshold of $\lin{n}$ is exponential in the norm of $\ph$. With a focus on lowering the regularity assumptions on $\ph$ rather than improving the threshold for $\ph$ smooth, this result was gradually improved by several authors -- see e.g. Mityagin~\cite{Mityagin:2004wn} and the references therein. The novelty of Theorem~\ref{ev-trap} consists in providing a threshold for $\lin{n}$ which is quadratic in the norm of $\ph$.

The proof is based on a \emph{Lyapunov-Schmidt decomposition} introduced by Kappeler \& Mityagin~\cite{Kappeler:1999er} -- see also \cite{Grebert:1998cz,Grebert:2001vm}: For the zero potential each $n\pi$, $n\in\Z$, is a double eigenvalue of $L$ with eigenfunctions $e_n^+ \defl (0,\e^{\ii n\pi x})$ and $e_n^- \defl (\e^{-\ii n\pi x},0)$. Thus, for a nonzero potential, provided $\abs{n}$ is sufficiently large, we expect exactly two eigenvalues which are close to $n\pi$ and whose eigenfunctions are close to the linear span of $e_{n}^{+}$ and $e_{n}^{-}$. This suggest to separate these modes from the others by a Lyapunov-Schmidt reduction.

To set the stage, we cover the complex plane with the closed strips
\[
  \Uf_n \defl \setdef{\lm}{\abs{\Re \lm-n\pi}\le \pi/2},
\]
and consider the Hilbert space of complex 2-periodic $L^2$-functions
\[
  L^{2}_{*}
   = \Pc_n\oplus \Qc_n
   = \operatorname{sp}\setd{e_n^+,e_n^-} \oplus
     \ob{\operatorname{sp}}\setdef{e_k^+,e_k^-}{k\neq n}.
\]
The orthogonal projections onto $\Pc_n$ and $\Qc_n$ are denoted by $P_n$ and $Q_n$, respectively.

To proceed, write the eigenvalue equation $Lf = \lm f$ in the form
\[
  A_\lm f = \Phi f,
\]
with the unbounded operators
\[
  A_\lm \defl \lm - \mat[\bigg]{\,\ii & \\  & -\ii\,}\ddx,
   \qquad \Phi \defl \mat[\bigg]{ & \phm\\ \php & }.
\]
Since $A_\lm$ leaves the spaces $\Pc_n$ and $\Qc_n$ invariant, by writing
\[
  f = u + v = P_nf + Q_nf,
\]
we can decompose the equation $A_\lm f = \Phi f$ into the two equations
\[
  A_\lm u = P_n\Phi(u+v),\qquad
  A_\lm v = Q_n \Phi(u+v),
\]
called the $P$- and the $Q$-equation, respectively.

We first consider the $Q$-equation on $\Uf_n$. One checks that for $m\neq n$
\[
  \min_{\lm\in \Uf_n} \abs{\lm-m\pi} \ge \abs{n-m} \ge 1,
\]
hence it follows with $A_\lm e_m^\pm = (\lm-m\pi)e_m^\pm$ that the restriction of $A_\lm$ to $\Qc_n$ is boundedly invertible for all $\lm\in \Uf_n$. Therefore, multiplying the $Q$-equation from the left by $\Phi A_\lm^{-1}$, gives
\[
  \Phi v = T_n\Phi(u+v),
\]
with $T_n \defl \Phi A_\lm^{-1} Q_n$. The latter identity may be written as
\[
  (I - T_n)\Phi v = T_n\Phi u,
\]
hence solving the $Q$ equation amounts to inverting $(I-T_{n})$.

We consider operator norms induced by \emph{shifted weighted norms} \cite{Kappeler:2001hsa,Poschel:2011iua}. Let $H_{i}^{w}$ denote the space of complex 2-periodic functions $u=\sum_{m\in\Z} u_{m}\e_{m}$ equipped with the $i$-shifted $H^w$-norm given by
\[
\n{u}_{w;i}^2
 \defl \n{u\e_i}_{w}^2 
 = \sum_{m\in\Z} w_{m+i}^{2} \abs{u_m}^2,
 \qquad \e_{m}(x) \defl \e^{\ii m\pi x}.
\]
On the space $H_{i,c}^w \defl H_{i}^{w}\times H_{i}^{w}$ of 2-periodic functions with values in $\C^{2}$,
\[
  f=(f_-,f_+) = \sum_{n\in\Z} (f^{-}_{n}e_{n}^{-} + f^{+}_{n}e_{n}^{+}) = 
  \sum_{n\in\Z} (f^{-}_{n}\e_{-n}^{-},f^{+}_{n}\e_{n}),
\]
the $i$-shifted norm is defined by
\[
  \n{f}_{w;i}^2
   \defl \n{f_-}_{w;-i}^2 + \n{f_+}_{w;i}^2 
     = \sum_{m\in\Z} w_{m+i}^{2} \paren[\big]{\abs{f_m^-}^2 + \abs{f_m^+}^2}.
\]

\begin{lemma}
\label{Tn-est}
If $\ph\in H_c^w$ with $w\in\Mc$, then for any $n,i\in\Z$ and any $\lm\in \Uf_n$,
\[
  T_n = \Phi A_\lm^{-1}Q_n\colon H_{i,c}^{w}\to H_{-i,c}^{w}
\]
is bounded and satisfies the estimate
\[
  \n{T_n f}_{w;i} \le 2\n{\ph}_w\n{f}_{w;- i}.
\]
\end{lemma}

\begin{proof}
Write $T_n f = \Phi g$ with $g=A_\lm^{-1}Q_nf$. Since the restriction of $A_\lm$ to $\Qc_n$ is boundedly invertible, the function
\[
  g
   = A_\lm^{-1}Q_nf
   = \sum_{m\neq n} \biggl(\frac{f_m^-}{\lm-m\pi}e_m^- + \frac{f_m^+}{\lm-m\pi}e_m^+ \biggr)
   = (g_-,g_+)
\]
is well defined. By Hölder's inequality we obtain for the weighted $\ell^1$-norm
\[
  \n{g_+\e_{-i}}_{\ell^1_w} 
  =
  \sum_{m\neq n} \frac{w_{m-i}\abs{f_m^+}}{\abs{\lm-m\pi}}
  \le \biggl(\sum_{m\neq n} \frac{1}{\abs{n-m}^2}\biggr)^{1/2}
      \n{f_+}_{w;-i}
  < 2\n{f_+}_{w;-i},
\]
uniformly for $\lm\in \Uf_n$; similarly $\n{g_-\e_{i}}_{\ell^1_w} \le 2\n{f_-}_{w;i}$. Since
\[
  \n{T_nf}_{w;i}^2
   = \n{\Phi g}_{w;i}^2
   = \n{\phm g_+\e_{-i}}_w^2+\n{\php g_-\e_i}_w^2,
\]
we can use standard inequalities for the convolution of sequences to obtain
\[
  \n{T_nf}_{w;i}^2
   \le \n{\ph}^2_w\left(\n{g_+\e_{-i}}_{\ell^1_w}^2+\n{g_-\e_i}_{\ell^1_w}^2\right)
   \le 4\n{\ph}_w^2\n{f}_{w;-i}^2.\qed
\]
\end{proof}

Note that $T_{n}f$ is estimated with a shifted $H^{w}$-norm where the sign of the shift is opposite to the sign of the shifted $H^{w}$-norm of $f$. This fact will be crucial in the following. In particular, $T_{n}^{2}$ is bounded  on $H_{i,c}^{w}$ and it turns out that $\n{T_n^2}_{w;n}=o(1)$ as $\abs{n}\to~\infty$. Using a Neumann series, we then obtain the bounded invertibility of $(I-T_n)$ for $\abs{n}$ sufficiently large, which solves the $Q$-equation. To exploit the regularity assumption $\ph\in H^{1}_{c}$ of Theorem~\ref{ev-trap}, we restrict ourselves to the subclass $\Mc^{1}$ of weights which have at least a linearly growing factor -- see also \cite{Djakov:2006ba,Grebert:2001vm} for weights with factors $\lin{n}^{\dl}$, $0 < \dl < 1/2$.

\begin{lemma}
\label{Tn2-est}
If $\ph\in H_c^w$ with $w\in\Mc^{1}$, then for any $n\in\Z$ and any $\lm\in \Uf_n$,
\[
  \n{T_n^2}_{w;n} \le \frac{4}{\lin{n}}\n{\ph}_w^2.
\]
\end{lemma}

\begin{proof}
As in the preceding lemma, write $T_n^2 f = \Phi g$ with
\[
  g \defl
  A_\lm^{-1}Q_n \Phi A_\lm^{-1} Q_n f.
\]
A straightforward computation yields
\[
  g = \sum_{k,l \neq n}\p*
    {
      \frac{\ph_{k+l}^-}{\lm-k\pi} \frac{f^+_l}{\lm-l\pi}e_k^-
      +
      \frac{\ph_{k+l}^+}  {\lm-k\pi} \frac{f^-_l}{\lm-l\pi}e_k^+
    }
    = (g_-,g_+),
\]
and our aim is to estimate the weighted $\ell^1$-norms $\n{g_+\e_{-n}}_{\ell^1_w}$ and $\n{g_-\e_{n}}_{\ell^1_w}$. By assumption $w=\lin{n}\cdot v$ with some submultiplicative weight $v$, hence
\[
  w_{k-n} \le \frac{\lin{k-n}}{\lin{k+l}\lin{l+n}}w_{k+l}w_{-l-n},\qquad k,l\in\Z.
\]
Consequently, for any $\lm\in \Uf_{n}$
\[
  \n{g_+\e_{-n}}_{\ell^1_w} \le \sum_{k,l\neq n} \frac{\lin{k-n}}{\lin{k+l}\abs{n-k}\lin{l+n}\abs{n-l}} w_{k+l}\abs{\ph^{+}_{k+l}}w_{-l-n}\abs{f_{l}^{-}},
\]
and with Cauchy-Schwarz and Young's inequality for convolutions,
\[
  \phantom{\n{g_+\e_{-n}}_{\ell^1_w}} \le \p[\Bigg]
    {
      \sum_{k,l\neq n}
      \frac{\lin{k-n}^{2}}{\lin{k+l}^2\abs{n-k}^{2}\lin{l+n}^2\abs{n-l}^2}
    }^{1/2}
    \n{\ph}_w\n{f_-}_{w;-n}.
\]
One further checks that
\[
  \sum_{k\neq n} \frac{\lin{k-n}^{2}}{\lin{k+l}^2\abs{n-k}^{2}} \le 32/5,\qquad
  \sum_{l\neq n} \frac{1}{\lin{l+n}^2\abs{n-l}^2} \le \frac{5/2}{\lin{n}^{2}}.
\]
Hence, we obtain for $\n{g_+\e_{-n}}_{\ell^1_w}$ and similarly for $\n{g_-\e_{n}}_{\ell^1_w}$,
\[
  \n{g_+\e_{-n}}_{\ell_w^1}
   \le \frac{4}{\lin{n}}\n{\ph}_w\n{f_-}_{w;-n},\quad
  \n{g_-\e_{n}}_{\ell_w^1}
   \le \frac{4}{\lin{n}}\n{\ph}_w\n{f_+}_{w;n}.
\]
Finally, with $\n{T_n^2 f}_{w;n} = \n{\Phi g}_{w;n}$, this gives
\begin{align*}
  \n{T_n^2f}_{w;n}^2 
   \le \n{\ph}_w^2\p*{\n{g_+\e_{-n}}_{\ell_1^w}^2 + \n{g_-\e_{n}}_{\ell_1^w}^2}
   \le \frac{16}{\lin{n}^2} \n{\ph}_w^4 \n{f}_{w;n}^2.\qed
\end{align*}
\end{proof}

Consequently, $T_n^2$ is a $1/2$-contraction if $\lin{n}\ge 8\n{\ph}_w^2$. In view of
\[
  \hat{T}_n \defl (I-T_n)^{-1} = (I+T_n)(I-T_n^2)^{-1},
\]
one then finds a unique solution
\[
  \Phi v = \hat{T}_nT_n \Phi u
\]
of the $Q$-equation. In turn, as $I+\hat{T}_{n}T_{n} = \hat{T}_{n}$, the $P$-equation yields
\[
  A_\lm u
   = 
  P_n (I + \hat{T}_nT_n)\Phi u
   =
  P_n\hat{T}_n\Phi u.
\]
Writing the latter as
\[
  S_n u = 0,\qquad S_n \colon \Pc_{n}\to \Pc_{n},\quad u \mapsto (A_\lm - P_n\hat{T}_n\Phi)u,
\]
we immediately conclude that there exists the following relationship.

\begin{lemma}
For $\ph\in H_{c}^{1}$ and $\lin{n} \ge 8\n{\ph}_{1}^{2}$ a complex number $\lm\in \Uf_{n}$ is an eigenvalue of $L$ if and only if the determinant of $S_{n}$ vanishes.~
\end{lemma}

\begin{proof}
Suppose $Lf = \lm f$, then by the preceding discussion $S_{n}u = 0$.
Conversely, define for any $u\in \Pc_{n}$,
\[
  v = A_{\lm}^{-1}Q_{n}\hat{T}_{n}\Phi u \in \Qc_{n}.
\]
Then $\Phi v = T_{n}\hat{T}_{n}\Phi u$ is well defined, and it follows with $\hat{T}_{n} = I + T_{n}\hat{T}_{n}$ that
\[
  A_{\lm}v = Q_{n}\hat{T}_{n}\Phi u = Q_{n}\Phi u + Q_{n}T_{n}\hat{T}_{n}\Phi u = Q_{n}\Phi(u+v),
\]
so the $Q$-equation is automatically satisfied. Moreover, if $S_{n}u = 0$, then
\[
  A_{\lm}u = P_{n}\hat{T}_{n}\Phi u = P_{n}(I + T_{n}\hat{T}_{n})\Phi u = P_{n}\Phi(u+v).
\]
Hence also the $P$-equation is satisfied, and $\lm$ is an eigenvalue of $L$ with eigenfunction $f = u+v$.\qed
\end{proof}

Recall that $P_n$ is the orthogonal projection onto the two-dimensional space $\Pc_n$. The matrix representation of an operator $B$ on $\Pc_n$ is given by
\[
  \p*{\lin{Be_n^\pm,e_n^\pm}}_{\pm,\pm}.
\]
Therefore, we find for $S_n$ the representation
\begin{align*}
  A_\lm
   = 
  \mat[\bigg]{
  \lm - n\pi              \\
             & \lm - n\pi 
  },
  \qquad
  P_n\hat{T}_n\Phi =
  \mat[\bigg]{
  a_n^+ & b_n^+\\
  b_n^- & a_n^-
  },
\end{align*}
with the coefficients of the latter matrix given by
\[
  a_n^\pm \defl \lin{\hat{T}_n\Phi e_n^\pm,e_n^\pm},\qquad
  b_n^\pm \defl \lin{\hat{T}_n\Phi e_n^\mp,e_n^\pm}.
\]

We point out that these coefficients depend on $\lm$ and $\ph$. It has been observed in \cite{Djakov:2006ba,Kappeler:2009kp} that these coefficients reflect certain symmetries of the Fourier coefficients of $\ph$. We only need the fact that $a_{n}^{+}$ and $a_{n}^{-}$ coincide.

\begin{lemma}
\label{coeff-sym}
Suppose $\ph\in H_c^1$ and $\lin{n}\ge 8\n{\ph}_1^2$, then for any $\lm\in \Uf_n$,
\[
  a_n^+(\lm) = a_n^-(\lm) \equiv a_{n}(\lm).
\]
\end{lemma}

\begin{proof}
Recall that $T_n = \Phi A_\lm^{-1} Q_n$, and denote by $\ob{B}u \defl \ob{B\ob{u}}$ the complex conjugation of operators.
From evaluating the bounded diagonal operators $A_\lm^{-1}$ and $Q_n$ at $e_m^\pm$, and using the identity $\ob{e_m^\pm} = Pe_m^\mp$, we conclude
\[
  (A_\lm^{-1})^*
   = P\ob{A_\lm^{-1}}P = A_{\ob\lm}^{-1},\quad 
  Q_n^*
   = P\ob{Q_n}P = Q_n,\qquad\quad 
  P \defl \Bigl( \begin{smallmatrix}
  &\; 1\\ 1&
\end{smallmatrix} \Bigr).
\]
Since $A_\lm^{-1}$ leaves $\Qc_n$ invariant, and $P^2 = I$, we find $(A_\lm^{-1}Q_n)^* = P\overline{A_\lm^{-1}Q_n}P$. With $\Phi^* = P\ob{\Phi}P$ this gives
\[
  (T_n\Phi)^* = \Phi^* (A_\lm^{-1} Q_n)^* \Phi^* = P \ob{T_n\Phi}P.
\]
Inspecting the Neumann expansion of $\hat{T}_n\Phi$ yields  $(\hat{T}_n\Phi)^* = P\overline{\hat{T}_n\Phi}P$, thus
\begin{align*}
  a_n^+
   &= \lin{\hat{T}_n\Phi e_n^+,e_n^+}
   = \lin{e_n^+,(\hat{T}_n\Phi)^*e_n^+}\\
   &= \lin{Pe_n^+,\overline{\hat{T}_n\Phi}Pe_n^+}
   = \lin{\hat{T}_n\Phi e_n^-,e_n^-} = a_n^-.\qed
\end{align*}
\end{proof}

It follows that $S_n$ may be written as
\[
  S_n(\lm) = 
  \mat[\bigg]{
  \lm-n\pi - a_n & -b_n^+\\
  -b_n^-         & \lm-n\pi - a_n
  }.
\]
As $T_n$ and $\Phi$ are anti-diagonal while $I$ and $T_n^2$ are diagonal, all even terms $\lin{T_{n}^{2k}\Phi e_{n}^{+},e_{n}^{+}}$ in the expansion of $a_{n}$ vanish. Using $\hat{T}_n = (I+T_n)(I-T_n^2)^{-1}$ we thus conclude
\[
  a_n = \lin{\hat{T}_n\Phi e_n^+,e_n^+} = \lin{T_n(I-T_n^2)^{-1}\Phi e_n^+,e_n^+}.
\]
On the other hand, all odd terms in the expansion of $b_{n}$ vanish, such that
\[
  b_n^\pm-\ph_{2n}^\pm
   = \lin{(\hat{T}_n-I)\Phi e_n^\mp,e_n^\pm}
   = \lin{T_n^2(I-T_n^2)^{-1}\Phi e_n^\mp,e_n^\pm}.
\]

We introduce the following notion for the sup-norm of a complex valued function on a domain $U\subset\C$,
\[
  \abs{f}_U \defl \sup_{\lm\in U} \abs{f(\lm)}.
\]

\begin{lemma}
\label{coeff-bounds}
If $\ph\in H_c^w$ with $w\in\Mc^{1}$, then for any $\lin{n}\ge 8\n{\ph}_w^2$ the coefficients $a_n$ and $b_n^\pm$ are analytic functions on $\Uf_n$ with bounds
\[
  \abs{a_n}_{\Uf_n} \le \frac{1}{\lin{n}}\n{\ph}_w^2,
  \qquad 
  w_{2n}\abs{b_n^\pm - \ph^\pm_{2n}}_{\Uf_n} \le \frac{8}{\lin{n}}\n{\ph}_w^2\n{\ph_{\!\pm}}_{w}.
\]
\end{lemma}

\begin{proof}
Since $\n{T_n^2}_{w;n}\le 1/2$, the series expansions of $a_n$ and $b_n^\pm$ converge uniformly on $\Uf_n$ to analytic functions.
Let $u=(I-T_n^2)^{-1}\Phi e_n^+$, then
\[
  \n{u}_{w;n}
   \le  \n{(I-T_n^2)^{-1}}_{w;n}\n{\Phi e_n^+}_{w;n}
   \le 2\n{\phm\e_{n}}_{w;-n} = 2\n{\phm}_{w},
\]
and with the series expansion $u=\sum_{m\in\Z} u_m e_{m}^-$ we may write
\[
  a_n
   = \lin{T_nu,e_n^+}
   = \sum_{m\neq n} \frac{\ph^-_{n+m}}{\lm-m\pi}u_m.
\]
As $\abs{n-m}\le \abs{n}$ implies $\abs{n+m}\ge 2\abs{n}-\abs{n-m}\ge \abs{n}$, this gives
\begin{align*}
  \abs{a_n}_{\Uf_{n}}
   &\le
    \sum_{m\neq n} 
    \frac{1}{\lin{n+m}^{2}\abs{n-m}} w_{n+m}\abs{\ph_{n+m}^-} \cdot w_{n+m}\abs{u_m}\\
   &\le \frac{1}{\lin{n}}\n{\phm}_w\n{u}_{w;n}
   \le \frac{2}{\lin{n}}\n{\phm}_w^2.
\end{align*}
Using the representation $a_{n} = \lin{T_{n}(I-T_{n}^{2})\Phi e_{n}^{-},e_{n}^{-}}$ we similarly obtain
\[
  \abs{a_n}_{\Uf_{n}} \le \frac{2}{\lin{n}}\n{\php}_w^2.
\]
Summing both estimates up gives the first bound. To obtain the second bound, we note that $b_n^- -\ph_{2n}^- = \lin{T_n^2u,e_{n}^-}$. Since $\lin{f,e_n^{-}} = \lin{f\e_{-n},e_{2n}^{-}}$ for any function $f\in L^{2}_{c}$, we conclude
\[
  w_{2n}\abs{\lin{T_n^2u,e_{n}^-}}
  \le \n{T_n^2u}_{w;n}
  \le \frac{8}{\lin{n}}\n{\ph}_w^2\n{\phm}_{w}.
\]
The proof for $b_n^+$ is the same.\qed
\end{proof}

In consequence, the determinant of $S_n$
\[
  \det S_n = (\lm-n\pi - a_n)^2 - b_n^+b_n^-
\]
is analytic on $\Uf_{n}$ and close to $(\lm-n\pi)^2$ provided $\lin{n}$ is sufficiently large.

\begin{lemma}
\label{Sn-roots}
Let $\ph\in H^1_c$, then for any $\lin{n}\ge 8\n{\ph}_1^2$, the determinant of $S_n$ has exactly two complex roots $\xi_+$, $\xi_-$ in $\Uf_n$, which are contained in the disc
\[
  D_n \defl \setd[\bigg]{\abs{\lm-\pi n} \le
  \frac{\n{\ph}_{1}^{2}}{\lin{n}}
  + \frac{\sqrt{2}\n{\ph}_{1}}{\lin{2n}}}
  \subset \setd[\bigg]{\abs{\lm-n\pi}\le \frac{\pi}{5}},
\]
and satisfy
\[
  \abs{\xi^{+}-\xi^{-}}^{2} \le 6\abs{b_{n}^{+}b_{n}^{-}}_{\Uf_{n}}.
\]
\end{lemma}
\begin{proof}
The estimates of the preceding lemma give for $\lin{n}\ge 8\n{\ph}_{1}^{2}$
\[
  \abs{a_{n}}_{\Uf_{n}} \le \frac{\n{\ph}_{1}^{2}}{\lin{n}},\qquad
  \lin{2n}^{2}\abs{b_{n}^{+}b_{n}^{-}}_{\Uf_{n}} \le \left(1 + \frac{8\n{\ph}_{1}}{\lin{n}}\right)^{2}
  \n{\php}_{1}\n{\phm}_{1},
\]
where we used $w_{2n}\abs{b_{n}^{\pm}} \le \n{\ph}_{w} + w_{2n}\abs{b_{n}^{\pm}-\ph_{2n}^{\pm}}$.
Therefore,
\[
  \abs{a_n}_{\Uf_n} + \abs{b_n^{+}b_{n}^{-}}_{\Uf_n}^{1/2}
   \le \inf_{\lm\in\,\Uf_{n}\setminus D_{n}}\abs{\lm-n\pi}
    = \frac{\n{\ph}_{1}^{2}}{\lin{n}}
     + \sqrt{2}\frac{\n{\ph}_{1}}{\lin{2n}}
   \le \pi/5.
\]
It follows from Rouche's Theorem that the function $h = \lm - n\pi - a_{n}$ has a single root in $D_n$, just as $\lm-n\pi$. Furthermore, $h^2$ and $\det S_n$ have the same number of roots in $D_n$, namely two when counted with multiplicity, while $\det S_n$ clearly has no root in $\Uf_n\setminus D_n$.

To estimate the distance of the roots, we write $\det S_n=g_+g_-$ with
\[
  g_{\pm} = \lambda-\pi n - a_n \mp \sigma_n, \qquad \sigma_n = \sqrt{b_n^+b_n^-},
\]
where the branch of the root is immaterial. Each root $\xi$ of $\det S_n$ is either a root of $g_+$ or $g_-$, respectively, and thus satisfies $\xi = \pi n + a_n(\xi) \pm \sigma_n(\xi)$.
Therefore,
\begin{align*}
  \abs{\xi_+-\xi_-}
   &\le \abs{a_n(\xi_+)-a_n(\xi_-)} + \abs{\sigma_n(\xi_+)\pm\sigma_n(\xi_-)}\\
   &\le \abs{\partial_{\lm}a_{n}}_{D_{n}}\abs{\xi_+-\xi_-} + 2\abs{\sigma_n}_{\Uf_n}.
\end{align*}
Since $\dist(D_{n},\partial \Uf_{n}) \ge \pi/2 - \pi/5$, the claim follows with Cauchy's estimate
\[
  \abs{\partial_\lambda a_n}_{D_n}
   \le \frac{\abs{a_n}_{\Uf_n}}{\dist(D_n, \partial \Uf_n)}
   \le \frac{1/8}{\pi/2-\pi/5} \le \frac{1}{6}.\qed
\]
\end{proof}

\begin{proof}[Proof of Theorem~\ref{ev-trap}.]
For each $\lin{n}\ge 8\n{\ph}_1^2$ Lemma~\ref{Sn-roots} applies giving us the two roots $\xi_+$ and $\xi_-$ of $\det S_n$ in $D_n\subset \Uf_n$. As the strips $\Uf_n$ cover the complex plane, and since $\lm_{n}^{\pm} \sim n\pi$ asymptotically as $n\to\pm\infty$ while there are no periodic eigenvalues in $\bigcup_{\lin{n} \ge 8\n{\ph}_{1}^{2}} (\Uf_n\setminus D_{n})$, it follows by a standard counting argument that these roots have to be the periodic eigenvalues $\lm_n^\pm$. In turn, the remaining eigenvalues have to be contained in the strip
\[
  \setdef*{\lm\in\C}{\abs{\Re \lm} \le (8\n{\ph}_{1}^{2}-1/2)\pi}.
\]
To obtain the estimate for the imaginary part, suppose $f$ is a $L^{2}_{c}$ normalised eigenfunction for $\lm$, then
\[
  2\Im \lm = \lm-\ob{\lm} = \lin{Lf,f}-\lin{f,Lf} = \lin{(L-L^{*})f,f}.
\]
Further, using the $L^{\infty}$-estimate $\n{g}_{\infty} \le \sqrt{2}\n{g}_{1}$ for $g\in H^{1}$, we find
\[
  \n{(L-L^{*})f}_{0} \le \sqrt{2}\n{\php-\ob{\phm}}_{1}\n{f}_{0} \le 2\n{\ph}_{1}.
\]
This completes the proof of the theorem.\qed
\end{proof}

Incidentally, we obtain the following estimate for the gap lengths, which we will use in section~\ref{s:act-west}.

\begin{proposition}
\label{gap-est}
Suppose $\ph\in H^{w}_{c}$ with $w\in\Mc^{1}$, then for any $\lin{N}\ge 8\n{\ph}_{w}^{2}$,
\[
  \sum_{\abs{n}\ge N} w_{2n}^{2}\abs{\gm_{n}(\ph)}^{2}
   \le 6\n{R_{N}\ph}_{w}^{2} +  \frac{1152}{\lin{N}}\n{\ph}_{w}^{6},
\]
where $R_{N}\ph = \sum_{\abs{n}\ge N} (\ph_{2n}^{-}\e_{-2n},\ph_{2n}^{+}\e_{2n})$.
If, in addition, $\ph$ is in the complex neighbourhood $W$ of $L^{2}_{r}$, then
\[
  \sum_{n\in\Z} w_{2n}^{2}\abs{\gm_{n}(\ph)}^{2}
   \le 265\pi^{2}w^{2}[16\n{\ph}_{w}^{2}](1+\n{\ph}_{w}^{2})\n{\ph}_{w}^{2}.
\]
\end{proposition}

\begin{proof}
By Lemma~\ref{Sn-roots} we have for any $\lin{n} \ge 8\n{\ph}_{w}^{2}$ the estimate
\[
  \abs{\gm_{n}}^{2}
   =   \abs{\lm_{n}^{+}-\lm_{n}^{-}}^{2}
   \le 6\abs{b_{n}^{+}b_{n}^{-}}_{\Uf_{n}}
   \le 3\p{\abs{b_{n}^{+}}_{\Uf_{n}}^{2} + \abs{b_{n}^{-}}_{\Uf_{n}}^{2}}.
\]
Using $\abs{b_{n}^{\pm}}_{\Uf_{n}} \le \abs{\ph_{2n}^{\pm}} + \abs{b_{n}^{\pm}-\ph_{2n}^{\pm}}_{\Uf_{n}}$ we thus find for any $\lin{N}\ge 8\n{\ph}_{w}^{2}$,
\begin{align*}
  &\frac{1}{6}\sum_{\abs{n}\ge N} w_{2n}^{2}\abs{\gm_{n}(\ph)}^{2}\\
   &\qquad \le \sum_{\abs{n}\ge N} w_{2n}^{2}(\abs{\ph_{2n}^{+}}^{2} + \abs{\ph_{2n}^{-}}^{2}
           +   \abs{b_{n}^{+}-\ph_{2n}^{+}}_{\Uf_{n}}^{2} 
       		+ \abs{b_{n}^{-}-\ph_{2n}^{-}}_{\Uf_{n}}^{2}).
\end{align*}
Further by Lemma~\ref{coeff-bounds}, $w_{2n}^{2}\abs{\ph_{2n}^{\pm}-b_{n}^{\pm}}_{\Uf_{n}}^{2} \le 64\lin{n}^{-2}\n{\ph}_{w}^{4}\n{\ph_{\pm}}_{w}^{2}$, hence
\[
  \frac{1}{6}\sum_{\abs{n}\ge N} w_{2n}^{2}\abs{\gm_{n}(\ph)}^{2}
   \le \n{R_{N}\ph}_{w}^{2} + 64\n{\ph}_{w}^{6}\sum_{\abs{n}\ge N} \frac{1}{\lin{n}^{2}},
\] 
and the first claim follows with $\sum_{\abs{n}\ge N} 1/\lin{n}^{2} \le 3/\lin{N}$.

If additionally $\ph\in W$, then each gap is contained in its isolating neighbourhood $U_{n}$. Those are disjoint complex discs centred on the real line, whose diameters for $\abs{n} < N$ sum up to at most $(2N-1)\pi$ by Theorem~\ref{ev-trap}. Therefore,
\[
  \sum_{\abs{n} < N} w_{2n}^{2}\abs{\gm_{n}(\ph)}^{2}
   \le w^{2}_{2N-2}\p[\Bigg]{\sum_{\abs{n} < N}\abs{\gm_{n}(\ph)}}^{2}\\
   \le w^{2}_{2N-2}(2N-1)^{2}\pi^{2},
\]
and choosing $N+1 \ge 8\n{\ph}_{w}^{2} > N$ gives the second claim.\qed
\end{proof}

If, for $\lin{n}\ge 8\n{\ph}_{w}^{2}$, we use
\[
  w_{2n}\abs{b_{n}^{\pm}}
   \le \n{\ph_{\pm}}_{w} + \frac{8}{\lin{n}}\n{\ph}_{w}^{2}\n{\ph_{\pm}}
   \le 2\n{\ph_{\pm}}_{w},
\]
then we obtain the \emph{individual gap estimate}
\begin{align}
  \label{i-gap}
  w_{2n}\abs{\gm_{n}(\ph)} \le 4\n{\ph}_{w}.
\end{align}

\section{Estimating the Actions}
\label{s:act-est}

As an immediate corollary to the localisation obtained in the previous section, we obtain the following quantitative estimate of the high-level actions.

\begin{proposition}
\label{In-Jn-est}
If $\ph\in H_{r}^{1}$, then for $\lin{n}\ge 8\n{\ph}_1^2$ and $m\ge 0$,
\[
  J_{n,2m+1} = \zeta_{n,m}^{2m} I_{n},\qquad
  \abs{\zeta_{n,m}-n\pi} \le \frac{\n{\ph}_{1}^{2}}{\lin{n}} + \frac{\sqrt{2}\n{\ph}_{1}}{\lin{2n}}.
\]
In particular, if $n\neq 0$ and $\lin{n}\ge 8\n{\ph}_1^2$, then
\[
  2^{-m}\lin{2n\pi}^{2m}I_{n} \le 4^{m}J_{n,2m+1} \le \lin{2n\pi}^{2m}I_{n},
\]
while the remaining actions for all $\lin{n} < 8\n{\ph}_{1}^{2}$ satisfy
\[
   4^{m}\abs{J_{n,2m+1}} \le (16\pi)^{2m}\n{\ph}_1^{4m} I_n.
\]
\end{proposition}

\begin{proof}
Recall from \eqref{zt-In-Jn} that $J_{n,2m+1} = \zt_{n,m}^{2m}I_{n}$  with $\zt_{n,m}\in[\lm_{n}^{-},\lm_{n}^{+}]$. Provided $\lin{n}\ge8\n{\ph}_{1}^{2}$, the estimate of $\abs{\zt_{n,m}-n\pi}$ follows from Theorem~\ref{ev-trap}. If additionally $n\neq 0$, then $\lin{2n}\ge 3\lin{n}/2$ and hence $\abs{\zt_{n,m}-n\pi} \le 1/2$. In consequence,
\[
  \frac{1}{\sqrt{2}}\lin{2n\pi} \le  2\abs{\zt_{n,m}} \le \lin{2n\pi},\qquad n\neq 0.
\]
Conversely, if $\lin{n} < 8\n{\ph}_{1}^{2}$, then $\abs{\zt_{n,m}} \le 8\pi \n{\ph}_{1}^{2}$.\qed
\end{proof}

In the sequel we use Proposition~\ref{In-Jn-est} to obtain an estimate of
\[
  \n{I(\ph)}_{\ell^{1}_{2m}} = \sum_{n\in\Z} \lin{2n\pi}^{2m}\abs{I_{n}}
\]
in terms of $\sum_{n\in\Z} J_{n,2m+1}$ and a remainder depending solely on $\n{\ph}_{1}$. The trace formula and the polynomial structure of the Hamiltonians then allows us to obtain the first part of Theorem~\ref{act-sob-est}.

\begin{lemma}
\label{I-H}
For every $m\ge 1$,
\[
  \n{I(\ph)}_{\ell^{1}_{2m}}
   \le \lin{16\pi}^{2m}(1+\n{\ph}_{1})^{4m}\n{\ph}_{0}^2 + (-1)^{m+1}2^{m}\Hc_{2m+1}(\ph),
\]
uniformly for all $\ph\in H_r^{m}$.
\end{lemma}

\begin{proof}
Choose $N+1\ge 8\n{\ph}_{1}^2 > N$, then by Proposition~\ref{In-Jn-est}, the trace formula~\eqref{tf-k}, and the positivity of the actions
\[
  \sum_{\abs{n}> N} \lin{2n\pi}^{2m} I_n
    \le 8^{m}\sum_{n\in\Z} J_{n,2m+1}
    = (-1)^{m+1}2^{m}\Hc_{2m+1}.
\]
On the other hand, by our choice of $N$ and the trace formula~\eqref{tf-1}
\[
  \sum_{\abs{n} \le N} \lin{2n\pi}^{2m}I_n
   \le \lin{2N\pi}^{2m}\sum_{n\in\Z} I_n
   \le \lin{16\pi}^{2m}(1+\n{\ph}_{1})^{4m}\n{\ph}_{0}^2.\qed
\]
\end{proof}

We denote the two components of $\ph\in L^{2}_{r}$ by $\ph = (\psi,\ob{\psi})$, and to simplify notation write $\psi_{(m)} = \partial_{x}^{m} \psi$ such that
\[
  \int_{\T} \abs{\psi_{(m)}}^{2}\,\dx = \frac{1}{2}\n{\ph_{(m)}}_{0}^{2}.
\]
We further note that on $H_{r}^{m}$,
\[
  (-1)^{m+1}\Hc_{2m+1}(\ph)
   = 
  \int_\T \paren*{\abs{\psi_{(m)}}^2 + p_{2m}(\psi,\ob\psi,\ldots,\psi_{(m-1)},\ob \psi_{(m-1)})}\,\dx,
\]
with $p_{2m}$ being a homogenous polynomial of degree $2m+2$ when $\psi$, $\ob\psi$, and $\partial_x$ each count as one degree. Further, the degree of each monomial of $p_{2m}$ with respect to $\psi$ equals the degree with respect to $\ob\psi$ -- see Corollary~\ref{h-form} from the appendix.
Consequently, each monomial $\mathfrak{q}$ of $p_{2m}$ may be estimated by
\[
  \abs{\mathfrak{q}} \le c_{\mathfrak{q}} \abs{\psi}^{\mu_{0}}\abs{\psi_{x}}^{\mu_{1}}\dotsm\abs{\psi_{(m-1)}}^{\mu_{m-1}},
\]
with some positive constant $c_{\mathfrak{q}}$ and integers $\mu_{0},\dotsc,\mu_{m-1}$.  Since $\mathfrak{q}$ has degree $2m+2$ we have
\[
  \sum_{0\le i\le m-1} (1+i)\mu_{i} = 2m+2,
\]
and since the degree with respect to $\psi$ and $\ob{\psi}$ is the same, $\sum_{0\le i\le m-1} \mu_{i}$ is an even integer. Denote by $\Ic_{2m+2}\subset (\Z_{\ge0}/2)^{m}$  the set of all multi-indices $\mu = (\mu_{i})_{0\le i\le m-1}$ satisfying the constraints
\begin{align}
  \label{p-const}
  \sum_{0\le i\le m-1} (1+i)\mu_{i} = 2m+2,\qquad \abs{\mu} \defl \sum_{0\le i\le m-1} \mu_{i} 
  \;\in\;2\Z.
\end{align}
Then, we obtain the estimate
\begin{align}
 \abs{p_{2m}} \le \sum_{\mu\in\Ic_{2m+2}} c_{\mu} \abs{\mathfrak{q}_{\mu}},\qquad
  \label{p-est}
  \abs{\mathfrak{q}_{\mu}} = \abs{\psi}^{\mu_{0}}\abs{\psi_{x}}^{\mu_{1}}\dotsm\abs{\psi_{(m-1)}}^{\mu_{m-1}},
\end{align}
with positive constants $c_{\mu}$. This representation of $p_{2m}$ allows us to obtain detailed estimates of the Hamiltonians $\Hc_{2m+1}$.

\begin{lem}
\label{pm-Hm-est}
For any $m\ge 1$ and any $\ep > 0$ there exists $C_{\ep,m}$ so that
\[
  \int_{\T} \abs{p_{2m}}\,\dx
   \le \ep\nn{\partial_{x}^{m}\psi}^{2} + C_{\ep,m}(1+\nn{\psi}^{4m})\nn{\psi}^{2}.
\]
In particular,
\[
  \abs{\Hc_{2m+1}}
   \le \p*{1+\ep}\nn{\partial_{x}^{m}\psi}^{2} + C_{\ep,m}(1+\nn{\psi}^{4m})\nn{\psi}^{2}.
\]
\end{lem}

\begin{proof}
First note that $\abs{\mu} \ge 4$ for any $\mu\in\Ic_{2m+2}$. Indeed, under the constraint $\abs{\mu} = k$ for some fixed $k\ge 0$ the expression $\sum_{0\le i\le m-1} (1+i)\mu_{i}$ attains its maximum when $\mu_{m-1} = \abs{\mu}$ while all other coefficients vanish. In this case,
\[
  \sum_{0\le i\le m-1} (1+i)\mu_{i} = m\abs{\mu},
\]
and the right hand side is strictly less than $2m+2$ if $\abs{\mu} \le 2$. Therefore, $\abs{\mu}$ is an even integer strictly larger than $2$.

As a consequence, for any $\mu\in\Ic_{2m+2}$, there either exist two distinct nonzero coefficients $\mu_{k}$, $\mu_{l}$ with $0\le k,l\le m-1$ or $\mu_{k} = \abs{\mu}\ge 4$ for some $0\le k\le m-1$ while all other coefficients vanish. In the first case, using Cauchy-Schwarz and the $L^{\infty}$-estimate, we obtain
\begin{align*}
  \int_{\T} \abs{\mathfrak{q}_{\mu}} \,\dx
   &\le \int \p*{\prod_{0\le i\neq k,l}^{m-1} \abs{\psi_{(i)}}^{\mu_{i}}}
             \abs{\psi_{(k)}}^{\mu_{k}-1}
             \abs{\psi_{(l)}}^{\mu_{l}-1} 
             \abs{\psi_{(k)}}
             \abs{\psi_{(l)}} \,\dx\\
  &\le \p*{\prod_{0\le i\neq k,l}^{m-1} \n{\psi_{(i)}}_{L^{\infty}}^{\mu_{i}}}
       \n{\psi_{(k)}}_{L^{\infty}}^{\mu_{k}-1}
       \n{\psi_{(l)}}_{L^{\infty}}^{\mu_{l}-1} 
       \nn{\psi_{(k)}}\nn{\psi_{(l)}}.
\end{align*}
It follows from the generalized Gagliardo-Nierenberg inequality that for any integer $0\le i\le m$
\begin{equation}
  \label{int-ineq}
  \n{\psi_{(i)}}_{L^{\infty}} \les \n{\psi}_{m}^{\frac{i+1/2}{m}}\n{\psi}_{0}^{1-\frac{i+1/2}{m}},\qquad
  \n{\psi_{(i)}}_{0} \les \n{\psi}_{m}^{\frac{i}{m}}\n{\psi}_{0}^{1-\frac{i}{m}},
\end{equation}
where $a\les b$ means $a\le c \cdot b$ with $c$ being a multiplicative constant which is independent of $\psi$ and only depends on the parameters $i$ and $m$. We thus obtain
\begin{equation}
  \label{q-mu-est}
  \int_{\T} \abs{\mathfrak{q}_{\mu}} \,\dx \les 
  \n{\psi}_{m}^{\p*{\sum_{i=0}^{m-1} \frac{i+1/2}{m}\mu_{i}} - \frac{1}{m}}
  \n{\psi}_{0}^{\p*{\sum_{i=0}^{m-1} \p*{1-\frac{i+1/2}{m}}\mu_{i}} 
  +\frac{1}{m}}.
\end{equation}
In the second case where $\mu_{k} = \abs{\mu}\ge 4$ while all other coefficients of $\mu$ vanish, we get
\begin{align*}
  \int_{\T} \abs{\mathfrak{q}_{\mu}} \,\dx
   &\le \int \p*{\prod_{0\le i\neq k,l}^{m-1} \abs{\psi_{(i)}}^{\mu_{i}}}
             \abs{\psi_{(k)}}^{\mu_{k}-2}
             \abs{\psi_{(k)}}^{2} \,\dx\\
  &\le \p*{\prod_{0\le i\neq k,l}^{m-1} \n{\psi_{(i)}}_{L^{\infty}}^{\mu_{i}}}
       \n{\psi_{(k)}}_{L^{\infty}}^{\mu_{k}-2}
       \nn{\psi_{(k)}}^{2}.
\end{align*}
The interpolation inequality~\eqref{int-ineq} then yields estimate~\eqref{q-mu-est} also in this case.

Recall from~\eqref{p-const} that $\sum_{i=0}^{m-1} (i+1/2)\mu_{i} = 2m+2 - \abs{\mu} / 2$, hence
\begin{align*}
  \sum_{i=0}^{m-1} \frac{i+1/2}{m}\mu_{i} - \frac{1}{m}
  &= \frac{2m+2-\abs{\mu}/2}{m} - \frac{1}{m}
  = 2 - \frac{\abs{\mu}-2}{2m} < 2,
\end{align*}
where in the last line we used that $4 \le \abs{\mu} \le 2m+2$. Similarly,
\begin{align*}
  \sum_{i=0}^{m-1} \p*{1-\frac{i+1/2}{m}}\mu_{i} + \frac{1}{m}
   = \abs{\mu} - 2 +  \frac{\abs{\mu}-4}{2m} 
   = \p*{1+\frac{1}{2m}}(\abs{\mu}-2).
\end{align*}
Both identities together yield in view of~\eqref{q-mu-est}
\[
  \int_{\T} \abs{\mathfrak{q}_{\mu}} \,\dx \les
  \n{\psi}_{m}^{2 - \frac{\abs{\mu}-2}{2m}}
  \n{\psi}_{0}^{\p*{1+\frac{1}{2m}}(\abs{\mu}-2)}.
\]
Applying Young's inequality to the latter with
\[
  p  = \frac{2}{2 - \frac{\abs{\mu}-2}{2m}},\qquad 
  p' = \frac{4m}{\abs{\mu}-2},
\]
then finally gives for any $\ep > 0$
\begin{align*}
  \int_{\T} \abs{\mathfrak{q}_{\mu}} \,\dx 
  &\le 
  \ep \n{\psi}_{m}^{\p*{2 - \frac{\abs{\mu}-2}{2m}}p}
  +
  C_{\ep,m}\n{\psi}_{0}^{\p*{1+\frac{1}{2m}}(\abs{\mu}-2)p'}\\
  &= \ep \n{\psi}_{m}^{2} + C_{\ep}\n{\psi}_{0}^{4m + 2},
\end{align*}
where $C_{\ep,m}$ is an absolute constant independent of $\psi$. Since this estimate holds for any monomial of $p_{2m}$, the final claim follows with the fact that $\n{\psi}_{m}^{2} \le 2^{m-1}\nn{\partial_{x}^{m}\psi}^{2} + 2^{m-1}\nn{\psi}^{2}$.\qed
\end{proof}

\begin{proof}[Proof of Theorem~\ref{act-sob-est} (i)]
Recall from Lemma~\ref{I-H}, that
\[
  \n{I(\ph)}_{\ell^{1}_{2m}}
   \le \lin{16\pi}^{2m}(1+\n{\ph}_{1})^{4m}\n{\ph}_{0}^2 + (-1)^{m+1}2^{m}\Hc_{2m+1}(\ph),
\]
which together with the estimate of the Hamiltonian from Lemma~\ref{pm-Hm-est} with $\ep = 1$ yields
\[
  \n{I(\ph)}_{\ell_{2m}^{1}}
   \le
  2^{m}
  \n{\ph}_{m}^{2}
  +
  c_m^2(1+\n{\ph}_{1})^{4m}\n{\ph}_{0}^{2},\qquad m\ge 1,
\]
where $c_{m}$ is an absolute constant depending only on $m$.\qed
\end{proof}

\section{Estimating the Sobolev Norms}
\label{s:sob-est}

We now turn to the converse problem of estimating the Sobolev norms of the potential in terms of weighted norms of its actions on level one. Our starting point is the identity
\begin{align}
  \label{sob-rep}
  \frac{1}{2}\n{\ph_{(m)}}_0^2
   = 4^{m}\sum_{n\in\Z} J_{n,2m+1} - \int_\T p_{2m}\,\dx,\qquad m\ge 1,
\end{align}
which is obtained by combining Corollary~\ref{h-form} and the trace formula~\eqref{tf-k}.
The key step is estimating the actions $J_{n,2m+1}$ in terms of $I_{n}$. Subsequently, $p_{2m}$ is estimated by Lemma~\ref{pm-Hm-est}.
The main difficulty is to estimate the actions $J_{n,2m+1}$ below the threshold of $\lin{n}$ provided by Proposition~\ref{In-Jn-est}. Even though there are only finitely many of them, they cannot be controlled by the $L^{2}$-norm $\n{\ph}_{0}$ as one can translate the spectrum of $\ph$ without changing $\n{\ph}_{0}$. Instead, we use the $H^{1}$-norm $\n{\ph}_{1}$, and provide estimates of $\n{\ph}_{1}$ in terms of $I_{n}$ by separate arguments which were inspired by work of Korotyaev~\cite{Korotyaev:2005fb,Korotyaev:2010ft}.

\begin{lemma}
\label{H3-est}
Uniformly for all $\ph\in H_r^1$,
\[
  \Hc_3(\ph) - 2\Hc_1^2(\ph) \le \sum_{n\in\Z} (2n\pi)^2 I_n(\ph).
\]
In particular, $\Hc_{3}(\ph) \le \n{I(\ph)}_{\ell^{1}_{2}} + 2\n{I(\ph)}_{\ell^{1}}^2$ and
\[
  \frac{1}{3}\n{\ph}_1^2 \le \n{I(\ph)}_{\ell^{1}_{2}} + \n{I(\ph)}_{\ell^{1}}^2.
\]
\end{lemma}

\begin{proof}
As the Hamiltonians and the actions are continuous on $H^{1}_{r}$, it suffices to consider the case of a finite gap potential. Let $C_{r}$ denote a sufficiently large circle enclosing all open gaps of $\ph$, then the primitive $F$ of $\om$ defined in section~\ref{s:setup} is analytic outside $C_{r}$ and its Laurent expansion is given by Lemma~\ref{F-asy}. Thus, by the Residue Theorem,
\[
  \frac{1}{\pi} \int_{C_{r}} F^{3}(\lm)\,\dlm
   = \frac{3}{8\ii \pi} \int_{C_{r}} \frac{1}{\lm}(\Hc_{3}-2\Hc_{1}^{2}) \,\dlm
   = \frac{3}{4}(\Hc_{3}-2\Hc_{1}^{2}).
\]
The right hand side is real as $\ph$ is of real type, and
\[
  \Re \int_{C_{r}} F^{3}(\lm)\,\dlm
   = \sum_{n\in\Z} \int_{\lm_{n}^{-}}^{\lm_{n}^{+}} \Re\bigl(F^{3}(\lm-\ii 0)
   - F^{3}(\lm+\ii 0)\bigr)\ \dlm.
\]
Furthermore, by Lemma~\ref{F-prop} we have for $\lm_{n}^{-} < \lm < \lm_{n}^{+}$,
\[
  \Re\bigl(F^{3}(\lm-\ii 0) - F^{3}(\lm+\ii 0)\bigr) = -2 f_{n}^{3}(\lm) + 6 (n\pi)^{2} f_{n}(\lm),
\]
and since $f_{n}$ is nonnegative, we conclude with \eqref{Jn-fn}
\[
  \frac{1}{\pi} \int_{C_{r}} F^{3}(\lm)\,\dlm
   \le \sum_{n\in\Z} \frac{6}{\pi}\int_{\lm_{n}^{-}}^{\lm_{n}^{+}} (n\pi)^{2} f_{n}(\lm)\,\dlm
   = \frac{3}{4}\sum_{n\in\Z}(2n\pi)^{2}I_{n}.
\]
This proves the first claim. Note that $\lin{2n\pi}^{2} \le \frac{3}{2}(1+(2n\pi)^{2})$ for $n\in\Z$, so
\[
  \frac{1}{3}\n{\ph}_{1}^{2}
  \le \frac{1}{2}(\n{\ph_{x}}_{0}^{2} + \n{\ph}_{0}^{2})
   =
  \Hc_3 - \int_\T \abs{\psi}^4\,\dx + \Hc_{1}.
\]
Since $\int_{\T} \abs{\psi}^{4}\,\dx \ge \Hc_{1}^{2}(\ph)$ the second claim follows with
\[
  \frac{1}{3}\n{\ph}_{1}^{2}
   \le \Hc_3 - 2\Hc_1^2 + \Hc_1^2 + \Hc_{1}
   \le \n{I(\ph)}_{\ell^{1}_{2}} + \n{I(\ph)}_{\ell^{1}}^{2}.\qed
\]
\end{proof}

\begin{proof}[Proof of Theorem~\ref{act-sob-est} (ii).]
The case $m=1$ is an immediate corollary of the lemma above.
For the case $m\ge 2$ we find with~\eqref{sob-rep}
\begin{align*}
  \frac{1}{2}\n{\ph_{(m)}}_0^2
   = 
  \sum_{n\in\Z} 4^{m}J_{n,2m+1} - \int_\T p_{2m}\,\dx.
\end{align*}
Choosing $\ep = 1/4$ in Lemma~\ref{pm-Hm-est} then gives
\[
  \frac{1}{4}\n{\ph_{(m)}}_0^2
   \les 
  \sum_{n\in\Z} 4^{m}J_{n,2m+1} + \n{\ph}_{0}^{4m+2}
   \les 
  \sum_{n\in\Z} 4^{m}J_{n,2m+1} + \n{I(\ph)}_{\ell^{1}}^{2m+1},
\]
where we applied the trace formula~\ref{tf-1} in the last step. Finally, in view of Lemma~\ref{h-a-est} below
\begin{align*}
  \frac{1}{4}\n{\ph_{(m)}}_0^2
    \le 
  \n{I}_{\ell^{1}_{2m}} + c_{m}(1+\n{I}_{\ell^{1}_{2}})^{4m-3}\n{I}_{\ell^{1}}.\qed
\end{align*}
\end{proof}

\begin{lemma}
\label{h-a-est}
For each $m\ge 1$,
\begin{align*}
  \sum_{n\in\Z}
  4^{m}J_{n,2m+1}
   \le 
  \n{I(\ph)}_{\ell^{1}_{2m}} +
  (64\pi)^{2m}(1+\n{I(\ph)}_{\ell^{1}})^{2m-1}\n{I(\ph)}_{\ell^{1}_{2}}^{2m-1},
\end{align*}
uniformly for all $\ph\in H_r^{m}$.
\end{lemma}

\begin{proof}
Let $N+1 \ge 8\n{\ph}_1^2 > N$, then by Proposition~\ref{In-Jn-est}
\[
  \sum_{\abs{n} > N} 4^{m}J_{n,2m+1} \le \sum_{\abs{n}> N} \lin{2n\pi}^{2m} I_n,
\]
while for the remaining actions $J_{n,2m+1} = \tilde\zt_{n,m}^{2m-2}J_{n,3}$ and hence
\begin{align*}
  \sum_{\abs{n} \le N} 4^{m}J_{n,2m+1}
   &\le 
  (16\pi)^{2m-2}\n{\ph}_1^{4m-4} \sum_{\abs{n}\le N} 4J_{n,3}\\
   &\le 
  (64\pi)^{2m-2}(1+\n{I}_{\ell^{1}})^{2m-2}\n{I}_{\ell_{2}^{1}}^{2m-2}\sum_{\abs{n}\le N} 4J_{n,3}.
\end{align*}
By the trace formula \eqref{tf-k} and Lemma~\ref{H3-est} we finally obtain
\[
  \sum_{n\in\Z} 4J_{n,3}
   = \Hc_3 
   \le  \n{I(\ph)}_{\ell^{1}_{2}} + 2\n{I(\ph)}_{\ell^{1}}^{2}.\qed
\]
\end{proof}

\section{Actions, Weighted Sobolev Spaces, and Gap Lengths}
\label{s:act-west}

The case of estimating the action variables of $\ph$ in standard Sobolev spaces $H_{r}^{m}$ with integers $m\ge 1$ is somewhat special due to the presence of the trace formula~\eqref{tf-k}. When arbitrary weighted Sobolev spaces $H_{r}^{w}$ are considered, there is no identity known to exist relating $\n{\ph}_{w}$ to Hamiltonians of the \nls-hierarchy. Albeit, even in the case of weighted Sobolev spaces, the regularity properties of $\ph$ are well known to be closely related to the decay properties of the  gap lengths $\gm_{n}(\ph)$ -- see e.g. \cite{Djakov:2006ba,Kappeler:2009kp} and section~\ref{s:trap-sp}. Moreover,
\begin{align}
  \label{In-gmn}
  \frac{4I_{n}}{\gm_{n}^{2}} = 1 + \ell_{n}^{2}
\end{align}
is known to hold locally uniformly on $L^{2}_{r}$ and hence uniformly on bounded subsets of $H^{1}_{r}$ -- see \cite{Grebert:2014iq}. In this section we prove a quantitative version of \eqref{In-gmn} which is quadratic in $\n{\ph}_{1}$ on all of $H^{1}_{r}$. From this and the estimates of the gap lengths given in section~\ref{s:trap-sp} we then obtain Theorem~\ref{act-west}.

To set the stage, let $\ph\in W$ and recall from \eqref{Jn-om},
\[
  I_{n} = \frac{1}{\pi}\int_{\Gm_{n}} \lm \om = -\frac{1}{\pi}\int_{\Gm_{n}} (\lm_{n}^{\ld}-\lm)\om.
\]
Here the latter identity follows as $\om$ is closed around the gap. In the case $I_{n}\neq 0$, or equivalently $\gm_{n}\neq 0$, we shrink the contour $\Gm_{n}$ to the straight line $[\lm_{n}^{-},\lm_{n}^{+}]$ and insert the product representation \eqref{om} of $\om$, to obtain
\[
  I_{n} = -\frac{2}{\pi}\int_{\lm_{n}^{-}}^{\lm_{n}^{+}} 
  \frac{(\lm_{n}^{\ld}-\lm)^{2}}{\ii\vs_{n}(\lm)}\chi_{n}(\lm)\,\dlm,
  \qquad
  \chi_{n}(\lm) = \prod_{n\neq m} \frac{\lm_{m}^{\ld}-\lm}{\vs_{m}(\lm)}.
\]
It follows with the substitution $\lm = \tau_{n} + t\gm_{n}/2$ that
\[
  \frac{4I_{n}}{\gm_{n}^{2}} = \frac{2}{\pi}\int_{-1}^{1}
   \frac{(t-t_{n})^{2}}{\sqrt[+]{1-t^{2}}}\chi_{n}(\tau_{n}+t\gm_{n}/2) \,\dt,
  \qquad
  t_{n} = 2(\lm_{n}^{\ld}-\tau_{n})/\gm_{n}.
\]
By Lemma~\ref{ld-tau} there exists an open connected neighbourhood $\hat{W}\subset W$ of $L^{2}_{r}$ such that $\abs{\lm_{m}^{\ld}-\tau_{m}} \le \abs{\gm_{m}}$ for all $m\in\Z$, hence $\abs{t_{n}} \le 2$, and thus
\begin{align}
  \label{In-gmn-2}
  \abs*{\frac{4I_{n}}{\gm_{n}^{2}}} \le 9\abs{\chi_{n}}_{[\lm_{n}^{-},\lm_{n}^{+}]}.
\end{align}
To get a quantitative version of \eqref{In-gmn} we thus need a uniform estimate of $\chi_{n}$ .

\begin{lemma}
On $H^{1}_{c}\cap \hat{W}$  for any $\abs{n} \ge \lin{N} \ge 8\n{\ph}_{1}^{2}$,
\[
  \abs{\chi_{n}}_{[\lm_{n}^{-},\lm_{n}^{+}]}
   \le \e^{2}\left(\frac{\abs{n} + N+3/5}{\abs{n} - N-3/5}\right)
   \le 2^{9}(1+\n{\ph}_{1}^{2}).
\]
\end{lemma}

\begin{proof}
Suppose $\ph\in H^{1}_{c}\cap \hat{W}$ and choose $\lin{N} \ge 8 \n{\ph}_{1}^{2} > N$. For $\abs{n} \ge N$ we split the product $\chi_{n}$ into two parts,
\[
  \chi_{n}(\lm) = 
  \biggl(\prod_{{\abs{m} \le N}} \frac{\lm_{m}^{\ld}-\lm}{\vs_{m}(\lm)}\biggr)
  \biggl(\prod_{\atop{n\neq m}{\abs{m} > N}} \frac{\lm_{m}^{\ld}-\lm}{\vs_{m}(\lm)}\biggr).
\]
If $\abs{k} > N$, then by Theorem~\ref{ev-trap}
\[
  \abs{\lm_{k}^{\pm}-k\pi} \le \frac{\n{\ph}_{1}^{2}}{\lin{k}}
   + \frac{\sqrt{2}\n{\ph}_{1}}{\lin{2k}}  \le \pi/8,
\]
where we used that $\lin{2k} \ge 3\lin{k}/2$. Thus, for $\abs{m},\abs{n} > N$ and $\lm\in [\lm_{n}^{-},\lm_{n}^{+}]$,
\[
  \abs{\tau_{m}-\lm} \ge 2\abs{n-m}.
\]
Further, $\abs{\gm_{m}} \le 4\n{\ph}_{1}/\lin{2m}$ by the individual gap estimate {\eqref{i-gap}}, hence
\[
  \abs*{\frac{\gm_{m}/2}{\tau_{m}-\lm}}
   \le \frac{\n{\ph}_{1}}{\lin{2m}\abs{n-m}}
   \le 1/4.
\]
It follows that $\abs{\vs_{m}(\lm)} \ge \abs{\tau_{m}-\lm} - \abs{\gm_{m}}/2$. Moreover, $\abs{\lm_{m}^{\ld} - \tau_{m}} \le \abs{\gm_{m}}$ as $\ph\in\hat{W}$, thus with $(1+2r)/(1-r) \le 1+4r$ for $0\le r\le 1/4$, we conclude
\[
  \abs*{\frac{\lm_{m}^{\ld}-\lm}{\vs_{m}(\lm)}}
   \le 
  \frac{\abs{\tau_{m}-\lm} + \abs{\gm_{m}}}{\abs{\tau_{m}-\lm} - \abs{\gm_{m}}/2}
   \le 1 + \frac{\abs{\gm_{m}}}{\abs{n-m}}.
\]
It follows with Cauchy-Schwarz that
\begin{align*}
  \sum_{\atop{m\neq n}{\abs{m} > N}} \frac{\abs{\gm_{m}}}{\abs{n-m}}
   &\le \left(\sum_{\atop{m\neq n}{\abs{m} > N}} \frac{1}{\lin{2m}^{2}\abs{n-m}^{2}}\right)^{1/2}
   \left(\sum_{\abs{m} > N} \lin{2m}^{2}\abs{\gm_{m}}^{2}\right)^{1/2}\\
   &\le \frac{2}{\lin{2N+2}}\left(\sum_{\abs{m} > N} \lin{2m}^{2}\abs{\gm_{m}}^{2}\right)^{1/2},
\end{align*}
and by Proposition~\ref{gap-est}
\[
  \sum_{\abs{m} > N} \lin{2m}^{2}\abs{\gm_{m}}^{2}
  	\le 6\n{\Rc_{N}\ph}_{1}^{2} + 144\n{\ph}_{1}^{4}
    \le 4\lin{N}^{2}.
\]
Therefore, by the standard estimates for infinite products,
\[
  \prod_{\atop{\abs{m}> N}{m\neq n}} \abs*{\frac{\lm_{m}^{\ld}-\lm}{\vs_{m}(\lm)}}
   \le \exp\left(4\lin{N}/\lin{2N+2} \right)
   \le \e^{2}.
\]

To estimate the remaining part of the product we note that $\lm_{m}^{\ld}$ and $\lm_{m}^{\pm}$ are contained in the isolating neighbourhood $U_{m}$, which is a complex disc centred on the real line. Thus if $\lm\in [\lm_{n}^{-},\lm_{n}^{+}]$ and $n > N$, then
\[
  \abs{\lm_{m-1}^{\pm}-\lm} > \abs{\lm_{m}^{\ld}-\lm},\qquad m\le N,
\]
and consequently
\[
  \prod_{\abs{m} \le  N} \abs*{\frac{\lm_{m}^{\ld}-\lm}{\vs_{m}(\lm)}}
  = \abs*{\frac{\lm_{-N}^{\ld}-\lm}{\vs_{N}(\lm)}}
  \prod_{-N < m \le N} \abs*{\frac{\lm_{m}^{\ld}-\lm}{\vs_{m-1}(\lm)}} 
  \le \frac{\abs{\tau_{-N}-\lm} + \abs{\gm_{-N}}}{\abs{\abs{\tau_{N}-\lm} - \abs{\gm_{N}}/2}}.
\]
By Theorem~\ref{ev-trap} we have $\abs{\gm_{\pm N}} \le \pi/5$, as well as
\[
  \abs{\pm N - n}\pi - 2\pi/5 \le \abs{\tau_{\pm N}-\lm} \le \abs{\pm N - n}\pi + 2\pi/5.
\]
It follows that
\[
  \frac{\abs{\tau_{-N}-\lm} + \abs{\gm_{-N}}}{\abs{\abs{\tau_{N}-\lm} - \abs{\gm_{N}}/2}}
  \le
  \frac{\abs{n}+N+3/5}{\abs{n}-N-3/5}\le 5\lin{N}.
\]
Similarly, for $\lm\in [\lm_{n}^{-},\lm_{n}^{+}]$ with $n < -N$.
\end{proof}

\begin{proposition}
Suppose $\ph\in H^{1}_{c}\cap \hat{W}$, then for any $\abs{n} \ge 8\n{\ph}_{1}^{2}$,
\[
  \abs{I_{n}} \le 2^{11}(1+\n{\ph}_{1}^{2})\abs{\gm_{n}}^{2}.
\]
\end{proposition}

\begin{proof}
If $\gm_{n} = 0$, then $I_{n} = 0$ and the estimate clearly holds. If $\gm_{n}\neq 0$, then by \eqref{In-gmn-2} and the preceding lemma,
\[
  \abs*{{4I_{n}}/{\gm_{n}^{2}}} \le 9\abs{\chi_{n}}_{[\lm_{n}^{-},\lm_{n}^{+}]}
  \le 2^{13}(1+\n{\ph}_{1}^{2}).
\]
\end{proof}

\begin{proof}[Proof of Theorem~\ref{act-west}.]
Suppose $\ph\in H^{w}_{c}\cap \hat{W}$ and choose $N+1\ge 8\n{\ph}_{w}^{2} > N$. Then by the preceding proposition
\[
  \sum_{\abs{n} > N} w_{2n}^{2}\abs{I_{n}}
   \le 2^{11}\n{\ph}_{1}^{2}\sum_{\abs{n} > N} w_{2n}^{2}\abs{\gm_{n}}^{2},
\]
and the gap lengths may be estimated by Proposition~\ref{gap-est}
\[
  \sum_{\abs{n} > N} w_{2n}^{2}\abs{\gm_{n}}^{2} 
  \le 6\n{R_{N}\ph}_{w}^{2} + 144\n{\ph}_{w}^{4}
  \le 144(1+\n{\ph}_{w}^{2})\n{\ph}_{w}^{2}.
\]
In particular, the mapping
\[
  H^{w}_{c}\cap\hat{W} \to [0,\infty),\qquad \ph\mapsto \sum_{n\in\Z} w_{2n}^{2}\abs{I_{n}},
\]
is continuous.
Suppose $\ph$ is of real type, then the remaining actions for $\abs{n} \le N$ may be estimated by
\[
  \sum_{\abs{n} \le N} w_{2n}^{2}\abs{I_{n}}
   \le w^{2}_{2N}\sum_{n\in\Z} I_{n}
   \le w[16\n{\ph}_{w}^{2}]^{2}\n{\ph}_{0}^{2}.
\]
Since $w\in\Mc^{1}$ is growing with at least linear speed, we thus find on $H^{w}_{r}$
\[
  \sum_{n\in\Z} w_{2n}^{2} \abs{I_{n}} \le 2^{20} w[16\n{\ph}_{w}^{2}]^{2}\n{\ph}_{w}^{2}.
\]
For any nonzero potential $\ph$ this estimate extends by continuity to a complex neighbourhood of $\ph$ within $H^{w}_{c}$ with just the absolute constant doubled. On the other hand, sufficiently close to the zero potential we have $\n{\ph}_{w}^{2} < 1/8$. In this case we may choose $N=0$ such that
\[
  \sum_{n\neq 0} w_{2n}^{2}\abs{I_{n}} \le 2^{20}(1+\n{\ph}_{w})^{4}\n{\ph}_{1}^{2}.
\]
Consequently, on some sufficiently small open neighbourhood of $H_{r}^{w}$ in $H_{c}^{w}$,
\[
  \n{I(\ph)}_{\ell_{w}^{1}} \le c_{w}^{2}w[16\n{\ph}_{w}^{2}]^{2}\n{\ph}_{w}^{2},
\]
with a real constant $c_{w}$ depending only on $w_{0}$.
\end{proof}

\clearpage

\appendix

\section{\boldmath Analyticity of the primitive $F$}

\label{a:ana}

In this appendix we prove the analyticity of the primitive
\[
  F(\mu,\psi) = \frac{1}{2}\left(\int_{\lm_{0}^{-}(\psi)}^{\mu} \om(\lm,\psi) + \int_{\lm_{0}^{+}(\psi)}^{\mu}\om(\lm,\psi) \right),
\]
introduced in section~\ref{s:setup} with $\om = \frac{\dDl}{\sqrt[c]{\Dl^{2}-4}}\, \dlm$. The proof relies on the following observation.

\begin{lemma}
\label{ld-tau}
\begin{itemize}
\item[(i)]
Suppose $\ph\in W$ and $\gm_{n}(\ph) = 0$ for some $n\in\Z$, then there exists an open neighbourhood $V_{n}$ of $\ph$ such that
\[
  \abs{\lm_{n}^{\ld}-\tau_{n}} \le \abs{\gm_{n}}/2,\qquad \psi\in V_{n}.
\]
\item[(ii)]
For each $\ph\in L^{2}_{r}$ there exists an open neighbourhood $V\subset L^{2}_{c}$ such that
\[
  \abs{\lm_{m}^{\ld}-\tau_{m}} \le \abs{\gm_{m}},\qquad m\in\Z,\quad \psi\in V.
\]
\end{itemize}
\end{lemma}

\begin{proof}
(i) For any potential in $V_{n}=V_{\ph}$ we have
\[
  0 = \frac{1}{4}(\Dl^{2}-4)^{\ld}\bigg|_{\lm_{n}^{\ld}}
    = \left(2(\tau_{n}-\lm)\Dl_{n}
       - \left((\tau_{n}-\lm)^{2} 
           - {\gm_{n}^{2}}/{4}\right)\Dl_{n}^{\ld}\right)\bigg|_{\lm_{n}^{\ld}},
\]
where the function $\Dl_{n}$ is analytic on $\C\times W$ and given by
\[
  \Dl_{n}(\lm)
   \defl \frac{1}{\pi_{n}^{2}}\prod_{m\neq n} \frac{(\lm_{m}^{+}-\lm)(\lm_{m}^{-}-\lm)}{\pi_{m}^{2}}.
\]
The zeroes of $\Dl_{n}$ are precisely the eigenvalues $\lm_{m}^{\pm}$ for $m\neq n$, thus $\Dl_{n}$ does not vanish on $U_{n}\times V_{n}$ and we have $\abs{\Dl_{n}} \ge s > 0$ and $\abs{\Dl_{n}^{\ld}} \le r$ uniformly on $U_{n} \times V_{n}$. Since $\gm_{n}(\ph) = 0$, we may shrink $V_{n}$, if necessary, to the effect that
\[
  \abs{\gm_{n}(\psi)}r \le s,\qquad \psi\in V_{n}.
\]
To simplify notation put $f = 2(\tau_{n}-\lm)\Dl_{n}$ and $g = (\Dl^{2}-4)^{\ld}/4$. By Lemma~\ref{w-closed} $\gm_{n}(\psi) = 0$ implies $\lm_{n}^{\ld}(\psi) = \tau_{n}(\psi)$, hence we may assume $\gm_{n}(\psi) \neq 0$. In this case, we find on the boundary of the disc $D_{\psi}=\setd{\abs{\lm-\tau_{n}(\psi)}\le \abs{\gm_{n}(\psi)}/2}$,
\[
  \abs{f-g}_{\partial D_{\psi}}
   \le \abs{\gm_{n}(\psi)}^{2}r/2 < s\abs{\gm_{n}(\psi)}
   \le \abs{f}_{\partial D_{\psi}}.
\]
Thus, by Rouché's Theorem, $f$ and $g$ have the same number of roots in $D_{\psi}$, and since $\lm_{n}^{\ld}$ is the only root of $g$ contained in $U_{n}$, the claim (i) follows.

(ii) Recall from \cite{Grebert:2014iq} that for each $\ph\in L^{2}_{c}$ there exists a neighbourhood $V$ and an $M\ge 0$ such that
\[
  \sum_{\abs{m}\ge M} \abs{\lm_{m}^{\ld}-\tau_{m}}^{2}/\abs{\gm_{m}}^{4} \le 1,\qquad \psi\in V.
\]
After possibly increasing $M$ we have $\abs{\lm_{m}^{\ld}-\tau_{m}} \le \abs{\gm_{m}}$ on $V$ for all $\abs{m}\ge M$. Suppose that, in addition, $\ph\in L^{2}_{r}$. If $\gm_{m}(\ph) = 0$ for some $\abs{m} < M$, then we may shrink $V$ according to (i) to obtain the desired inequality. On the other hand, if $\gm_{m}(\ph) \neq 0$, then
\[
  \abs{\lm_{m}^{\ld}(\ph)-\tau_{m}(\ph)} \le \abs{\gm_{m}(\ph)}/2,
\]
and since $\lm_{m}^{\ld}$, $\tau_{m}$ and $\abs{\gm_{m}}$ are continuous on $W$, the desired inequality follows after possibly shrinking $V$.
\end{proof}

\begin{proposition}
\label{F-ana}
The mapping $F$ is analytic on $(\C\setminus\bigcup_{n\in\Z} U_{n})\times V_{\ph}$.
\end{proposition}

\begin{proof}
Take any $\nu\in \partial U_{0}$, and write the integral as
\[
  F(\mu) = \frac{1}{2}\left(\int_{\lm_{0}^{-}}^{\nu} \om + \int_{\lm_{0}^{+}}^{\nu}\om \right) + \int_{\nu}^{\mu} \om,
\]
where the admissible path of integration of $\int_{\lm_{0}^{\pm}}^{\nu}\om$ runs in $U_{0}\setminus[\lm_{0}^{-},\lm_{0}^{+}]$ except for its end points, and the one of $\int_{\nu}^{\mu} \om$ runs in $\C\setminus\bigcup_{n\in\Z} U_{n}$. Then, by Lemma~\ref{w-closed}, $\int_{\nu}^{\mu}\om$ is analytic on $(\C\setminus\bigcup_{n\in\Z} U_{n})\times V_{\ph}$. We now want to prove that $F_{\nu} \defl F(\nu)$ is analytic on $V_{\ph}$ for an arbitrary $\nu\in \partial U_{0}$.

Denote by $Z_{0} \defl \setdef{\psi\in V_{\ph}}{\gm_{0}^{2}(\psi) = 0}$ the analytic subvariety of $V_{\ph}$ of potentials with a collapsed zeroth gap. We first prove analyticity on the open set $V_{\ph}\setminus Z_{0}$, secondly continuity on all of $V_{\ph}$, and finally weak analyticity on $Z_{0}$. The analyticity on $V_{\ph}$ then follows by the general theory of analytic functions -- see \cite[Appendix A]{Kappeler:2003up}. Note that by the argument above, we may move $\nu$ on $\partial U_{0}$ after each of the three steps.

\textit{Analyticity on $V_{\ph}\setminus Z_{0}$}. Fix any $\psi^{*}$ in $V_{\ph}\setminus Z_{0}$. Since the eigenvalues $\lm_{0}^{\pm}(\psi^{*})$ are simple, there exist two functions $\rho_{1}$ and $\rho_{2}$, which are analytic on some neighbourhood $V\subset V_{\ph}\setminus Z_{0}$ of $\psi^{*}$ and satisfy the set equality $\setd{\rho_{1},\rho_{2}} = \setd{\lm_{0}^{-},\lm_{0}^{+}}$. Provided the straight line $[\rho_{1},\nu]$ is an admissible path, we can use the parametrisation
\[
  \lm_{t} = \rho_{1} + tz,\qquad z = \nu-\rho_{1},
\]
together with the product representation \eqref{om} of $\om$ to get
\begin{align}
  \label{F-int-rep}
  F_{\nu}
   = \int_{\rho_{1}}^{\nu} \om
   = -\ii\int_{0}^{1} 
     \frac{\lm_{0}^{\ld}-\lm_{t}}{\vs_{0}(\lm_{t})}
     \chi_{0}(\lm_{t})z\,\dt,
     \qquad \chi_{0}(\lm) = \prod_{m\neq 0} \frac{\lm_{m}^{\ld}-\lm}{\vs_{m}(\lm)}.
\end{align}
Note that $\chi_{0}$ is analytic on $U_{0}\times V_{\ph}$ -- see \cite[Section 12]{Grebert:2014iq}. Let $\tilde\gm_{0} = \rho_{2}-\rho_{1}$ and fix $\nu$ on $\partial U_{0}$ such that for some real $\sg(\psi^{*}) > 0$,
\[
  \nu = \rho_{1}(\psi^{*}) - \sg(\psi^{*})\tilde\gm_{0}(\psi^{*})/2.
\]
After possibly shrinking $V$, we find for the same fixed $\nu$
\[
  \nu = \rho_{1}(\psi) - \sg(\psi)\tilde\gm_{0}(\psi)/2,\qquad \psi\in V,
\]
with $\sg(\psi)$ possibly complex though $\Re \sg \ge \ep > 0$ on $V$. As $z = -\sg\tilde\gm_{0}/2$, we conclude
\[
  \abs{\tau_{0}-\lm_{t}} = \abs{1+t\sg}\cdot\abs{\gm_{0}/2} > \abs{\gm_{0}/2},\qquad 0 < t \le 1,
\]
uniformly on $V$. Therefore, the path $[\rho_{1},\nu]$ is admissible for all $\psi$ in $V$. In view of \eqref{s-root}, the root $\vs_{0}(\lm_{t})$ is continuous in $t \ge 0$, analytic on $V$ for $t > 0$, and satisfies the lower bound
\[
  \abs{\vs_{0}(\lm_{t})}
   \ge \abs{\gm_{0}/2} \sqrt[+]{(1+t\ep)^{2}-1}
   \ge \abs{\gm_{0}/2} \sqrt[+]{2t\ep}.
\]
It follows that the integrand $(\lm_{0}^{\ld}-\lm_{t})\chi_{0}(\lm_{t})z/\vs_{0}(\lm_{t})$ of \eqref{F-int-rep} is continuous on $(0,1]\times V$, analytic on $V$ for fixed $t>0$, and has an integrable majorant. In consequence, $F_{\nu}$ is analytic on $V$, and hence on all of $V_{\ph}\setminus Z_{0}$.

\textit{Continuity on $V_{\ph}$}. Clearly, $F_{\nu}$ is continuous on $V_{\ph}\setminus Z_{0}$, and its restriction
\[
  F_{\nu}\Big|_{Z_{0}} = -\ii\int_{\tau_{0}}^{\nu} \chi_{0}(\lm)\,\dlm,
\]
is continuous, too, as $\chi_{0}$ and $\tau_{0}$ are. To establish the continuity of $F_{\nu}$ on $V_{\ph}$ it thus suffices to show that $F_{\nu}(\psi^{k})\to F_{\nu}(\psi^{*})$ for every sequence $(\psi^{k})$ in $V_{\ph}\setminus Z_{0}$ converging to some $\psi^{*}$ in $Z_{0}$.

In view of Lemma~\ref{ld-tau}, we may without loss assume $\abs{\lm_{0}^{\ld}-\tau_{0}} \le \abs{\gm_{0}}/2$ for all $\psi^{k}$. Substituting $\lm_{t} = \tau_{0} - t\gm_{0}/2$ such that $\vs_{0}(\lm_{t})^{2} = -(1-t^{2})\gm_{0}^{2}/4$ gives
\[
  \abs*{\int_{\lm_{0}^{-}}^{\tau_{0}} \frac{\lm_{0}^{\ld}-\lm}{\ii \vs_{0}(\lm)}\chi_{0}(\lm)\,\dlm}
   \le
   \int_{0}^{1} \abs*{\frac{\lm_{0}^{\ld} - \tau_{0} + t\gm_{0}/2}{\sqrt[+]{1-t^{2}}}\chi_{0}(\lm_{t})}\,\dt
   = O(\abs{\gm_{0}}),
\]
where the implicit constant can be chosen uniformly in $k$. Therefore,
\[
  F_{\nu}(\psi^{k}) = -\ii\int_{\tau_{0}}^{\nu} \frac{\lm_{0}^{\ld}-\lm}{\vs_{0}(\lm)}\chi_{0}(\lm)\,\dlm\bigg|_{\psi^{k}} + o(1).
\]
We can choose $\nu^{*}$ on $\partial U_{0}$ such that the straight line $[\tau_{0},\nu^{*}]$ is admissible for any $\psi^{k}$. With the parametrisation
\[
  \lm_{t} = \tau_{0} + tz,\qquad z = \nu^{*}-\tau_{0},
\]
we then obtain
\[
  -\ii\int_{\tau_{0}}^{\nu^{*}} \frac{\lm_{0}^{\ld}-\lm}{\vs_{0}(\lm)}\chi_{0}(\lm)\,\dlm
   = -\ii\int_{0}^{1} \frac{\lm_{0}^{\ld}-\lm_{t}}{\vs_{0}(\lm_{t})}\chi_{0}(\lm_{t})z\,\dt.
\]
Since $\abs{z}$ is uniformly bounded below on $V$,
\[
  \frac{\abs{\gm_{0}/2}}{\abs{\tau_{0}-\lm_{t}}} = O(\abs{\gm_{0}/t}).
\]
In view of \eqref{s-root} for any $\dl > 0$ there exists a neighbourhood $V_{\dl}\subset V$ of $\psi^{*}$ such that the root $\vs_{0}(\lm_{t})$ is continuous and does not vanish on $[\dl,1]\times V_{\dl}$. Consequently, for all $\dl > 0$,
\[
  -\ii\int_{\dl}^{1}\frac{\lm_{0}^{\ld}-\lm_{t}}{\vs_{0}(\lm_{t})}\chi_{0}(\lm_{t})z\bigg|_{\psi^{k}}
  \to
  -\ii\int_{\dl}^{1}\chi_{0}(\lm_{t})z\bigg|_{\psi^{*}},\qquad k\to \infty.
\]
Let $\ep = \abs{\gm_{0}}/2\dl\abs{z}$, then after possibly shrinking $V_{\dl}$ we have $0 < \ep < 1$ and
\begin{align*}
  \int_{0}^{\dl} \abs*{\frac{\lm_{0}^{\ld}-\lm_{t}}{\vs_{0}(\lm_{t})} \chi_{0}(\lm_{t}) z} \,\dt
  &\le \int_{0}^{\dl}
  		\frac{t\abs{z} + \abs{\gm_{0}/2}}{\sqrt{\abs{t^{2}\abs{z}^{2}-\abs{\gm_{0}/2}^{2}}}}
		\abs*{\chi_{0}(\lm_{t})z}\,\dt\\
  &= O\left(\dl \int_{0}^{1} \frac{\sqrt[+]{\abs{t+\ep}}}{\sqrt[+]{\abs{t-\ep}}}\,\dt\right),
\end{align*}
where the implicit constant is uniform on $V_{\dl}$. One checks that the latter integral is uniformly bounded for $0\le \ep \le 1$, hence
\[
  \sup_{\psi\in V_{\dl}}
  \int_{0}^{\dl} 
  \abs*{\frac{\lm_{0}^{\ld} - \lm_{t}}{\vs_{0}(\lm_{t})}\chi_{0}(\lm_{t})z}\,\dt
  = O(\dl).
\]
It follows that
\[
  -\ii\int_{0}^{1}\frac{\lm_{0}^{\ld}-\lm_{t}}{\vs_{0}(\lm_{t})}\chi_{0}(\lm_{t})z\bigg|_{\psi^{k}}
  \to
  -\ii\int_{0}^{1}\chi_{0}(\lm_{t})z\bigg|_{\psi^{*}},\qquad k\to \infty,
\]
so $F_{\nu^{*}}$ is continuous at $\psi^{*}$ and hence on all of $V_{\ph}$.

\textit{Weak analyticity.} The restriction of $F_{\nu}$ to $Z_{0}$ is given by
\[
  F_{\nu}\bigg|_{Z_{0}} = -\ii\int_{\tau_{0}}^{\nu} \chi_{0}(\lm)\,\dlm,
\]
and since $\chi_{0}$ and $\tau_{0}$ are both analytic, the weak analyticity follows directly. This completes the proof of the analyticity of $F_{\nu}$ on $V_{\ph}$.
\end{proof}

\section{Properties of the \nls hierarchy}
\label{s:hamil}

In this appendix we recall some well known facts about the Hamiltonians of the \nls-hierarchy -- see \cite{Grebert:2014iq}. For $\ph=(\phm,\php)\in H_{c}^{k-1}$ the $k$th \nls Hamiltonian is given by,
\[
  \Hc_{k}(\ph) = \int_\T \ph_{-} u_k(\phm,\php,\ldots,\phm^{(k-1)},\php^{(k-1)})\,\dx,\qquad k\ge 1,
\]
where $u_1 = -\php$, and
\[
  u_{k+1} = u_k' + \phm \sum\nolimits_{l=1}^{k-1}u_{k-l}u_l,\qquad k\ge 1.
\]

\begin{lemma}
If $\ph\in H_{c}^{k-1}$, then
\[
  u_{k+1} = -\php^{(k)} + q_k(\phm,\php,\ldots,\phm^{(k-2)},\php^{(k-2)}),
\]
where $q_k$ is a homogeneous polynomial of degree $k+1$ when $\phm$, $\php$, and $\partial_x$ each count as one degree. Moreover, each term of $\phm q_k$ has at most degree $k-2$ with respect to $\partial_x$, and the degree with respect to $\phm$ equals the one with respect to $\php$.~
\end{lemma}

\begin{proof}
As is evident from their definition, the polynomials $u_k$ are homogeneous of degree $k$, and only contain derivatives of $\phm$ and $\php$ up to order $k-1$. Furthermore, $u_k(\phm,\lm\php) = \lm u_k(\lm\phm,\php)$ for all $\lm\in\C$, which completes the proof.
\end{proof}

In case of a smooth real type potential, that is $\ph = (\psi,\ob\psi)$, the odd Hamiltonians have the form
\[
  (-1)^{m+1}\Hc_{2m+1}(\ph)
   = 
  \int_\T \left(\abs{\psi^{(m)}}^2
   + \psi q_{2m}\right)\,\dx,\qquad m\ge1,
\]
where $q_{2m}$ depends on $\psi$, $\ob\psi$, and their derivatives up to order $2m-2$. Suppose  $\psi q_{2m}$ contains a monomial $\psi_{(m+n)}\mathfrak{q}(\psi,\ob\psi,\ldots,\psi_{(2m-2)},\ob\psi_{(2m-2)})$ with $n\ge 0$. Since this monomial has at most degree $2m-2$ with respect to $\partial_{x}$, it follows that $\mathfrak{q}$ contains at most $m-2-n$ derivatives. Thus we can integrate by parts to the effect that each factor of the monomial contains at most $m-1$ derivatives.

\begin{corollary}
\label{h-form}
For any $m\ge 1$ there exists a polynomial $p_{2m}$ such that
\[
  (-1)^{m+1}\Hc_{2m+1}(\ph)
   = 
  \int_\T \paren*{\abs{\psi_{(m)}}^2 + p_{2m}(\psi,\ob\psi,\ldots,\psi_{(m-1)},\ob \psi_{(m-1)})}\,\dx,
\]
for all $\ph=(\psi,\ob\psi)$ in $H_{r}^{m}$. The polynomial $p_{2m}$ is homogenous of degree $2m+2$ with respect to $\psi$, $\ob\psi$, and $\partial_x$, and the degree of each term of $p_{2m}$ with respect to $\psi$ equals the one with respect to $\ob\psi$.
\end{corollary}

\bibliography{bibliography}

\bigskip
\footnotesize

\textit{MSC 2010:} 35Q55, secondary 35P05, 34L40.
\par
\textit{Keywords:} defocusing NLS, Birkhoff normal form, periodic Dirac operators.

\bigskip

\textit{Institut für Mathematik}
\par
\textit{Universität Zürich, Winterthurerstrasse 190, CH-8057 Zürich}
\par
\textit{jan.molnar@math.uzh.ch}

\end{document}